\documentclass[11pt,a4paper,reqno]{amsart}
\usepackage{amsmath,amsthm,amsfonts,amssymb,amscd,amstext}
\usepackage[latin1]{inputenc}
\usepackage[dvips]{graphicx}
\usepackage{psfrag}
\usepackage{euscript}
\usepackage{a4wide}
\usepackage{mathrsfs}
\usepackage{xcolor}
\usepackage{mathtools}
\usepackage[colorlinks=true,linkcolor=blue, urlcolor=red, citecolor=blue, backref]{hyperref}	

\numberwithin{equation}{section}

\theoremstyle{plain}
\newtheorem{maintheorem}{Theorem}

\newtheorem{theorem}{Theorem}[section]
\newtheorem{proposition}[theorem]{Proposition}
\newtheorem{corollary}[theorem]{Corollary}
\newtheorem{lemma}[theorem]{Lemma}
\newtheorem{claim}[theorem]{Claim}

\theoremstyle{definition}
\newtheorem{remark}[theorem]{Remark}
\newtheorem{definition}{Definition}
\newtheorem{example}{Example}

\newcommand{\RR}{{\mathbb R}}

\newcommand{\diam}{\operatorname{diam}}
\renewcommand{\epsilon}{\varepsilon}

\newcommand{\AP}{\operatorname{AP}}

\newcommand{\intt}{\operatorname{int}}

\newcommand{\Hom}{\operatorname{Hom}}

\newcommand{\supp}{\operatorname{supp}}

\begin{document}

\title[Completely irregular set]{On the completely irregular set of maps with the shadowing property}


\author{M. Carvalho}
\author{V. Coelho}
\author{L. Salgado}

\address{Maria Carvalho, CMUP \& Departamento de Matem\'atica, Faculdade de Ci\^encias da Universidade do Porto \\
Rua do Campo Alegre s/n, 4169-007 Porto, Portugal.}
\email{mpcarval@fc.up.pt}

\address{Vin\'icius Coelho, Universidade Federal do Oeste da Bahia, Centro Multidisciplinar de Bom Jesus da Lapa\\
Av. Manoel Novais, 1064, Centro, 47600-000 - Bom Jesus da Lapa-BA-Brazil}
\email{viniciuscs@ufob.edu.br}

\address{Luciana Salgado, Universidade Federal do Rio de Janeiro, Instituto de
   Matem\'atica\\
   Avenida Athos da Silveira Ramos 149 Cidade Universit\'aria, P.O. Box 68530,
   21941-909 Rio de Janeiro-RJ-Brazil }
 \email{lsalgado@im.ufjr.br, lucianasalgado@ufrj.br}

%




\begin{abstract}
We prove that the completely irregular set is Baire generic for every non-uniquely ergodic transitive continuous map which satisfies the shadowing property and acts on a compact metric space without isolated points. We also show that, under the previous assumptions, the orbit of any completely irregular point is dense.
Afterwards, we analyze the connection between transitivity and the shadowing property, draw a few consequences of their joint action within the family of expansive homeomorphisms, and discuss several examples to test the scope of our results.
\end{abstract}


\date{\today}
\maketitle
\tiny
\tableofcontents
\normalsize


\section{Introduction} \label{sec:statement-result1}

Let $(X,d)$ be a compact metric space and $T\colon X \to X$ a continuous map. Given $x \in X$, denote by $\mathcal{O}_T(x)$ the orbit by $T$ of $x \in X$, and by $\overline{\mathcal{O}_T(x)}$ its closure in $X$. In what follows,  $C^0(X, \mathbb{R})$ stands for the set of real valued continuous maps on $X$ endowed with the supremum norm $\|\cdot\|_\infty$; $\mathcal{P}(X)$ denotes the space of Borel probability measures in $X$ with the weak$^*$-topology, $\mathcal{P}_{T}(X)$ is its subspace of $T$-invariant elements and $\mathcal{P}_{T}^{erg}(X)$ is the subset of ergodic measures of $\mathcal{P}_{T}(X)$.

Given $x \in X$, denote by $\delta_{x}$ the Dirac measure supported at $x$~and let $V_T(x) \subset\mathcal{P}_T(X)$ be the set of accumulation points
of the sequence of empirical measures $\big(\frac{1}{n}\,\sum_{j=0}^{n-1}\,\delta_{T^{j}(x)}\big)_{n \, \in \, \mathbb{N}}$.

A point $x \in X$ is called \emph{irregular} if there is a continuous function $\varphi\colon X \to \mathbb{R}$ such that the sequence $\Big(\frac{1}{n} \sum_{j=0}^{n-1} \varphi (T^{j}(x)) \Big)_{n \,\in \,\mathbb{N}}$ is not convergent. Given $\varphi \in C^0(X,\mathbb R)$, the set of \emph{$(T,\varphi)$-irregular points}, or points with historic behavior with respect to $\varphi$, is defined by
$${\mathcal I}(T,\varphi) \,=\,  \Big\{x \in X \colon \Big(\frac{1}{n} \sum_{j=0}^{n-1}  \varphi (T^{j}(x)) \Big)_{n \,\in \,\mathbb{N}} \text{ does not converge}\Big\}.$$
If we denote
\begin{align*}
& \mathcal{H}(X,T)\, = \,\big\{\varphi \in C^0(X,\mathbb{R})\colon \,{\mathcal I}(T,\varphi) \neq \emptyset\big\}
\\
& \mathcal{R}(X,T) \, = \, \big\{\varphi \in C^0(X,\mathbb{R})\colon \,{\mathcal I}(T,\varphi)\text{ is Baire generic in } X\big\}
\end{align*}
then the set of \emph{completely irregular points} of $T$ is precisely the intersection
$$\mathcal{C}\mathcal{I}(X,T) \,= \bigcap_{\varphi \,\in \,\mathcal{H}(X,T)} \,{\mathcal I}(T,\varphi).$$
In \cite[Corollary 2.1]{CCSV21}\label{maincorollary_compact}, in a joint work with P. Varandas, the authors established a necessary and sufficient condition for the set $\mathcal{C}\mathcal{I}(X,T)$ to be Baire generic whenever $T$ has a dense orbit. See Section~\ref{se:chaotic} for more details, where we reformulate this condition into a simpler statement.

Following \cite{HLT21v2,KOR2016}, we define the \emph{measure center} of $T$ by
$$\mathcal{M}(X,T) \,=\, \overline{\bigcup_{\mu \,\in\,  \mathcal{P}_T(X)} \supp \mu}$$
where $\supp \mu$ is the support of the probability measure $\mu$. Tian showed (cf. \cite[Theorem 4.3 and Theorem 4.5]{Tian}) that, if $T$ is not uniquely ergodic and satisfies the almost specification property, then
$$\mathcal{M}(X,T) \, \subseteq \bigcap_{x \,\in\, \mathcal{C}\mathcal{I}(X,T)} \omega(x)$$
where $\omega(x)$ stands for the set of accumulation points of $\mathcal{O}_T(x)$; if, in addition, $\mathcal{M}(X,T) = X$, then the set of completely irregular points is Baire generic in $X$ and every point in $\mathcal{C}\mathcal{I}(X,T)$ has a dense orbit by $T$. Our first result is a counterpart of \cite{DTY15,Tian}.

\begin{maintheorem}\label{theorem-A}
Let $(X,d)$ be a compact metric space without isolated points and $T\colon X \to X$ be a continuous map.
\begin{itemize}
\item[(a)] If $\mathcal{C}\mathcal{I}(X,T) \neq \emptyset$, then
$$\bigcap_{x \,\in\, \mathcal{C}\mathcal{I}(X,T)} \,\omega(x) \, \neq \, \emptyset.$$
\item[(b)] If $x,y \in \mathcal{C}\mathcal{I}(X,T)$, then
$$\#\, \big(V_{T}(y) \setminus \big\{\mu \in V_{T}(y)\colon \, \supp(\mu) \cap \omega(x) \neq \emptyset \big\}\big)\, \leq \,1.$$
\item[(c)] If $x,y \in \mathcal{C}\mathcal{I}(X,T)$ and $\mu, \nu \in V_{T}(y)$, then
$$\supp(\mu) \cap  \big(X \setminus \omega(x)\big) \, = \,\supp(\nu) \cap  \big(X \setminus \omega(x)\big).$$
\item[(d)] For every $x \in \mathcal{C}\mathcal{I}(X,T)$ there exists a $T$-invariant set $\mathbb{B}_{x} \subseteq \overline{\mathcal{O}_T(x)}$ such that $\mathbb{B}_{x}$ is Baire generic in $\overline{\mathcal{O}_T(x)}$ and $x \in \mathbb{B}_{x} \subseteq \mathcal{C}\mathcal{I}(X,T)$.
\end{itemize}
\end{maintheorem}

Examining the properties of the set of points $x \in X$ whose $V_T(x)$ is maximum, equal to $\mathcal{P}_T(X)$, we recover \cite[Theorem D(3)]{HLT21v2} and improve the information conveyed by the basins of attraction of invariant probability measures.


\begin{maintheorem}\label{theorem-B}
Let $(X,d)$ be a compact metric space without isolated points and $T\colon X \to X$ be a continuous map. If $T$ is not uniquely ergodic, $\mathcal{M}(X,T) = X$ and $\mathcal{P}_{T}^{erg}(X)$ is dense in $\mathcal{P}_{T}(X)$, then:
\begin{itemize}
\item[(a)] $\big\{x\in X\colon \, V_T(x) = \mathcal{P}_{T}(X)\big\}$ is Baire generic in $X$.
\smallskip
\item[(b)] $\mathcal{C}\mathcal{I}(X,T)$ is Baire generic in $X$.
\smallskip
\item[(c)] $T$ is transitive.
\end{itemize}
\end{maintheorem}

We refer the reader to Remark~\ref{rem:Big VT} where we prove item (a) of Theorem~\ref{theorem-B} under a different set of assumptions.

It is known that every topologically mixing homeomorphism $T\colon X \to X$ of a compact metric space $X$ with the shadowing property also satisfies the specification property (see \cite{DGS76}). Moreover, if $T$ satisfies the specification property (hence $T$ is transitive, see \cite{DGS76, KLO16}) and is not uniquely ergodic, then the set of completely irregular points is Baire generic in $X$ (cf. \cite[Corollary 1]{Tian}). Yet, there exist examples of transitive continuous not uniquely ergodic maps satisfying the shadowing property but which do not have the specification property (cf. \cite[section 5]{LO18}). In \cite{DTY15}, Dong, Tian and Yuan proved that, if $T$ is not uniquely ergodic, satisfies the asymptotic average shadowing property and its measure center coincides with $X$, then $T$ is transitive and the set of completely irregular points is Baire generic in $X$. Our next result establishes that the transitivity together with the shadowing property are in fact enough to guarantee that the completely irregular set is Baire generic in $X$, provided that $X$ has no isolated points and
$T$ is not uniquely ergodic.

\begin{maintheorem}\label{theorem-C}
Let $(X,d)$ be a compact metric space without isolated points and $T\colon X \to X$ be a transitive continuous map which is not uniquely ergodic and satisfies the shadowing property.
Then:
\begin{itemize}
\item[(a)] $\mathcal{H}(X,T) \,\neq\, \emptyset$.
\smallskip
\item[(b)] $T$ has positive topological entropy.
\smallskip
\item[(c)] $\mathcal{M}(X,T) \,=\, X$.
\smallskip
\item[(d)] $\bigcup_{\mu \,\in\, \mathcal{P}_{T}^{erg}(X)} \,\mathcal{B}(\mu)$ is dense in $X$.
\smallskip
\item[(e)] $\mathcal{C}\mathcal{I}(X,T)$ is Baire generic in $X$.
\smallskip
\item[(f)] $T$ has sensitivity to initial conditions.
\end{itemize}
\end{maintheorem}


Summoning the knowledge of \cite{WIN15} on set valued functions, we can add information on the set of continuity points of the map $W_{\varphi}$. More precisely, given a Baire metric space $(Y, d)$, a continuous function $T\colon Y \to Y$ and $\varphi \in C^b(Y,\mathbb R)$, define
\begin{eqnarray}\label{def:Wphi}
W_{\varphi}\colon  \quad Y &\to& \mathcal{K}([\inf \varphi, \sup \varphi]) \nonumber \\
\quad y &\mapsto& \bigcap\limits_{N=1}^{+\infty} \, \overline{\Big\lbrace  \frac{1}{n}\,\sum\limits_{j=0}^{n-1} \,\varphi(T^{j}(y))\colon \, n \geq N \Big\rbrace}.
\end{eqnarray}
where, given a metric space $(W,D)$, the set $\mathcal{K}(W)$ is the collection of all nonempty compact subsets of $W$ endowed with the Hausdorff metric $d_{H}$ defined by
$$d_{H}(A,B) \, = \,\max \big\{d_{0}(A,B),\, d_{0}(B,A)\big\} \quad \forall\, A, B \in \mathcal{K}(W)$$
where $d_{0}(A,B) \,=\, \sup_{b \,\in\, B} \, \inf\,\big\{D(a,b) \colon \, a \in A\big\}.$
The next result is a particular case of a more general statement (see Theorem~\ref{Theorem-Psi_f}) that we will prove in Section~\ref{se:pD}.

\begin{maintheorem}\label{theorem-D}
Let $(X,d)$ be a compact metric space, $T\colon X \to X$ be a transitive continuous map and $\varphi \in C^0(X,\mathbb R)$. Then the following assertions are equivalent:
\begin{enumerate}
\item[$(a)$] ${\mathcal I}(T,\varphi)$ is Baire generic in $X$.
\smallskip
\item[$(b)$] There exists a dense subset $\mathcal{Z}$ of $X$ such that $\mathcal{Z} \cap {\mathcal I}(T,\varphi) \neq \emptyset$ and $W_{\varphi}$ is lower semi-continuous in $\mathcal{Z}$.
\end{enumerate}
\end{maintheorem}

The remainder of the paper is organized as follows. In Section~\ref{se:def}, we include a short glossary of the main concepts we will use. The aforementioned Theorems~\ref{theorem-A}, \ref{theorem-B} and \ref{theorem-C} are proved in the ensuing sections, where we also compare them with results established in other references. In Section~\ref{se:ci-orbits} we provide information on the orbits of completely irregular points.
In Section~\ref{se:examples} we discuss several examples, recovering and extending known information about them. In Section~\ref{se:chaotic}, we present a dynamical condition that is both necessary and sufficient to ensure that the set of irregular points is Baire generic, which is used in Section~\ref{se:pD} to prove Theorem~\ref{theorem-D}.

\section{Definitions}\label{se:def}

Assume that $(X,d)$ is a compact metric space and $T\colon X \to X$ is a continuous map.

\subsection{Transitivity}

The map $T$ is \textit{transitive} if for every nonempty open sets $U,\,V \subset X$ there exists $n \in \mathbb{N}$ such that $U \cap T^{-n}(V) \neq \emptyset$. $T$ has a \textit{dense orbit} if there is $x \in X$ such that $\mathcal{O}_T(x)$ is dense in $X$. It is worthwhile observing that, if the metric space is compact and has no isolated points, then the map is transitive if and only if it has a dense orbit (see \cite[Theorem~1.4]{AC12}); similarly, if $T$ is surjective, then $T$ is transitive if and only if it has a dense orbit (see \cite[Theorem 5.9]{W78}). Denote by $\mathrm{Trans}(X,T)$ the set of points $x \in X$ whose orbit by $T$ is dense in $X$.

We say that $T$ is \emph{minimal} if every orbit by $T$ is dense in $X$. We note that there are examples of minimal continuous maps on compact metric spaces which are uniquely ergodic (like the minimal rotations of the unit circle,  \cite[Theorem 6.20]{W78}), and there also exist minimal continuous maps which are uniquely ergodic and satisfy the shadowing property (for instance, an adding machine on a Cantor set, \cite{Moo13, MaiYe}). In both cases, the set of irregular points is empty.

\subsection{Sensitivity}

We say that $T$ has \emph{sensitivity to initial conditions} if there exists $\varepsilon > 0$ such that, for every $x \in X$ and any $\delta > 0$, there is $z \in B(x,\delta)$ satisfying
$$\sup_{n\,\in\,\mathbb N} \, d(T^{n}(x), T^{n}(z)) > \varepsilon$$
where $B(x,\delta)$ stands for the ball centered at $x$ with radius $\delta$.

A fundamental result provided by Auslander and Yorke claims that a minimal homeomorphism of a compact metric space is either equicontinuous or sensitive to initial conditions  (cf. \cite{AY80}; see \cite{GLM20, HKZ18} for a similar result). From this dichotomy, combined with fact that a minimal equicontinuous homeomorphism is uniquely ergodic (cf. \cite[Proposition 2.3]{B76}), we deduce that, whenever $X$ has no isolated points and $T\colon X \to X$ is a minimal homeomorphism which is not uniquely ergodic, then $T$ has sensitivity to initial conditions. Later, Akin et. al established that, if $T$ is transitive, though not minimal, and the set of minimal points is dense in $X$, then $T$ has sensitivity to initial conditions (cf \cite[Theorem 2.5]{AAB}). Recently, Moothathu showed that, if $T$ is a transitive homeomorphism satisfying the shadowing property, then either $T$ has sensitivity to initial conditions or $T$ is equicontinuous (cf. \cite[Corollary 1]{Moo11}).

\subsection{Shadowing}

A sequence $\{x_n\}_{n\in\mathbb{N}}\subset X$ is called a \textit{$\delta$-pseudo-orbit} of $T$ if
$$d(T(x_n),\,x_{n+1}) \,<\, \delta \quad \forall \, n \in \mathbb{N}.$$
A $\delta$-pseudo-orbit $\{x_n\}_{n\in\mathbb{N}}$ is said to be \textit{$\varepsilon$-shadowed} if there exists $y \in X$ such that
$$d(T^n(y),\,x_n) \,<\, \varepsilon \quad \forall\, n \in \mathbb{N}.$$
A map $T\colon X\to X$ has the \emph{shadowing property} if for any $\varepsilon > 0$ there exists $\delta>0$ such that any $\delta$-pseudo-orbit is $\varepsilon$-shadowed.

We note that the shadowing property is $C^0$ generic in the space of homeomorphisms of any compact manifold, and a similar result holds for $C^0$ generic continuous maps (cf. \cite{Pilyugin}).

\subsection{Specification}
According to Bowen~\cite{Bowen}, $T$ satisfies the \emph{specification property} if for any $\varepsilon > 0$ there exists $N(\varepsilon) \in \mathbb{N}$ such that for every $k \in \mathbb{N}$, any points $x_1, \dots, x_k$ in $X$, any sequence of positive integers $n_1, \dots, n_k$ and every choice of positive integers $L_1, \dots, L_k$ with $L_i \geqslant N(\varepsilon)$, there exists a point $x_0$ in $X$ such that
$$d\big(T^j(x_0),\,T^j(x_1)\big) \leqslant \varepsilon \quad \quad \forall \, 0 \leqslant j \leqslant n_1$$
and
$$d\big(T^{j + n_1 + L_1 + \dots + n_{i-1} + L_{i-1}}(x_0),\, T^j(x_i)\Big) \leqslant \varepsilon \quad \quad \forall \, 2 \leqslant i \leqslant k \quad \forall\, 0\leqslant j\leqslant n_i.$$

It is known that full shifts on finitely many symbols satisfy the specification property. Besides, factors of maps with the specification property also enjoy this property (cf. \cite{DGS76}). The importance of the specification property is illustrated by the fact that it guarantees that the set of probability measures supported on periodic orbits is dense in $\mathcal{P}_{T}(X)$  (cf. \cite{DGS76}), thus $\mathcal{P}_{T}^{erg}(X)$ is dense in $\mathcal{P}_{T}(X)$.

\subsection{Almost periodicity}

A subset $A$ of $X$ is called $T$-\emph{invariant} if $T(A) \subseteq A$; $A$ is said to be \emph{minimal} if it is nonempty, closed and $T$-invariant, and no proper nonempty subset of $A$ has these three properties altogether. A point $x \in X$ is said to be \emph{almost periodic} by $T$ if, for any open neighborhood $U$ of $x$ and every $n \in \mathbb{N}$, there exist $N \in \mathbb{N}$ and $k \in [n, n+N]$ such that $T^k(x) \in U$. Denote the set of almost periodic points by $T$ by $\AP(X,T)$ and the set of periodic points of $T$ by $\mathrm{Per}(X,T)$. In \cite{B66}, Birkhoff showed that $x$ is almost periodic by $T$ if and only if the closure of its orbit is a minimal set (and so such an $x$ is also called a \emph{minimal point}). We note that, by Zorn's Lemma, the set of minimal points is not empty. 
Clearly, the set $\mathrm{Per}(X,T)$ is a subset (possibly empty) of $\AP(X,T)$; moreover, $\mathrm{Per}(X,T) \subseteq  \AP(X,T) \subseteq \mathcal{M}(X,T)$.

In \cite{Gr72}, one may find examples of subshifts with positive entropy such that every point is minimal. We refer the reader to \cite{DT18} and \cite{Moo13} for other interesting relations between the set $\AP(X,T)$ and the entropy. It is known  (cf. \cite[Proposition 2]{S74}) that, for continuous map with the specification property acting on a compact metric space, the set of periodic points is dense in $X$; hence $\AP(X,T)$ is dense as well. The density of $\AP(X,T)$ with respect to other subspaces has recently received more attention (cf. \cite{HTW19, Moo11, KOR2016}). In particular, it has been proved that the set of minimal points is dense in the non-wandering set for dynamical systems with the shadowing property (cf. \cite[Corollary 1]{Moo11}), and that the set of minimal points is dense in the measure center of every dynamical system with the almost specification property (cf. \cite[Theorem 5.2]{KOR2016}).

\section{Proof of Theorem \ref{theorem-A}}\label{se:proof-A}

In what follows, given a set $E \subset X$ and $\delta>0$ we will denote by $E^{\delta}$ the union of the balls centered at points of $E$ with radius $\delta$, that is,
$$E^{\delta} \,=\, \bigcup_{x\, \in\, E} \,B(x, \delta).$$

\noindent $(a)$ We start by showing that, if the set $\mathcal{C}\mathcal{I}(X,T)$ of completely irregular points of $T$ is nonempty, then $\bigcap_{x \,\in\, \mathcal{C}\mathcal{I}(X,T)} \overline{\mathcal{O}_T(x)} \neq \emptyset$. Suppose, by contradiction, that the latter set is empty. Since $X$ is compact,
there exist $N \in \mathbb{N}$ and $x_{1}, \cdots, x_{N} \in \mathcal{C}\mathcal{I}(X,T)$ such that $\bigcap_{j=1}^{N} \overline{\mathcal{O}_T(x_{j})} = \emptyset$. So, there is $\delta > 0$ such that $\overline{\mathcal{O}_T(x_{1})\,}^{\delta} \cap \overline{\mathcal{O}_T(x_{j})} = \emptyset $ for all $j \in \{2,\cdots,N\}$. Since $\mathcal{C}\mathcal{I}(X,T)$ is not empty, we may find a continuous function $\varphi\colon X \to \mathbb{R}$ such that $x_{1},\cdots, x_{N}$ belong to
$\mathcal{I}(\varphi,T)$. Take a continuous function $\xi\colon X \to \mathbb{R}$ such that
$$\xi_{\mid_{\overline{\mathcal{O}_T(x_{1})}}} \, = \,1 \quad \quad \text{ and } \quad \quad \xi_{\mid_{X \setminus \overline{\mathcal{O}_T(x_{1})\,}^{\delta}}} \, = \,0$$
and consider the continuous map $\psi\colon X \to \mathbb{R}$ defined by $\psi(x) = \varphi(x)\,\xi(x)$. Therefore, $x_{1} \in \mathcal{I}(\psi,T)$, but $x_{2},...,x_{n}$ are not in $\mathcal{I}(\psi,T)$. This contradicts the choice of $x_{1},\cdots,x_{n} \in \mathcal{C}\mathcal{I}(X,T)$. Thus, $\bigcap_{x \,\in \,\mathcal{C}\mathcal{I}(X,T)} \overline{\mathcal{O}_T(x)} \neq \emptyset$.

\begin{lemma}\label{lem_top}
Let $F$ be a nonempty finite set of points of $X$. Then
\begin{equation}\label{eq:top}
\bigcap_{x \,\in\, F} \,\overline{\mathcal{O}_T(x)} \neq \emptyset \quad \Rightarrow \quad \bigcap_{x \,\in\, F} \,\omega(x) \neq \emptyset.
\end{equation}
\end{lemma}

\begin{proof}
The proof will be done by induction on the cardinality of $F$. If $\# F = 1$, then the statement holds trivially by the compactness of $X$. Suppose now that, for a fixed positive integer $n$ and any set $F \subset X$ with cardinality smaller or equal to $n$, the assertion \eqref{eq:top} is valid. Take a set $G \subset X$ with $n+1$ elements such that $\bigcap_{x \,\in \,G} \overline{\mathcal{O}_T(x)}\neq \emptyset$.

Consider the following partial order in $G$, which we denote by $\leq_{*}$: given $a, b \in G$,
$$a \,\leq_{*}\, b \quad \Leftrightarrow \quad \omega(a) \,\subseteq\, \omega(b).$$
We say that \emph{$a$ and $b$ are comparable} if $a \leq_{*}b$ or $b \leq_{*} a$. Recall that a \emph{total order} is a partial order under which every pair of elements is comparable, and a \emph{chain} of $G$ is a subset $H$ of $G$ that is a totally ordered set. Note that a set with one element is always a chain.

Let $C_{1}, \cdots, C_{K}$ be the chains of $G$. Observe that $K \leq n+1$ and $G = \bigcup_{j=1}^{K} C_{j}$. In each chain $C_{j}$ there exists a minimal element $z_{j}$, that is, $z_{j} \leq_{*} x$ for every $x \in C_{j}$. Therefore,
$$\bigcap_{x \,\in \,G} \,\omega(x) \,=\, \bigcap_{j=1}^{K} \,\omega(z_{j}).$$
We are left to show that $\bigcap_{j=1}^{K} \,\omega(z_{j}) \neq \emptyset.$

Suppose that $K < n+1$. Then $\{z_{1},\cdots, z_{K}\}$ is a subset of $X$ such that $\bigcap_{j=1}^{K} \overline{\mathcal{O}_T(z_{j})} \neq \emptyset$, since
$$\emptyset \,\neq\,  \bigcap_{x \,\in\, G} \overline{\mathcal{O}_T(x)}  \,\subseteq\,  \bigcap_{j=1}^{K} \overline{\mathcal{O}_T(z_{j})}.$$
By the hypothesis of induction, we conclude that
$$\emptyset \,\neq\, \bigcap_{j=1}^{K} \omega(z_{j}) \,\neq \,\emptyset.$$

If $K = n+1$, then every chain $C_j$ of $G$ contains only one element, say $z_j$. Consequently, $\mathcal{O}_T(z_{i}) \cap  \mathcal{O}_T(z_{j}) = \emptyset$ and $\mathcal{O}_T(z_{i}) \cap  \omega(z_{j}) = \emptyset$, for every $i \neq j \in \{1, \cdots, K\}$. As $\overline{\mathcal{O}_T(x)} = \mathcal{O}_T(x) \cup \omega(x)$ for any $x$ in $X$, we obtain
$$\emptyset \,\neq\,  \bigcap_{x \,\in\, G} \overline{\mathcal{O}_T(x)} \,=\, \bigcap_{j=1}^{K} \overline{\mathcal{O}_T(z_{j})} \, = \, \bigcap_{j=1}^{K} \omega(z_{j}).$$
This completes the proof of Lemma \ref{lem_top}.
\end{proof}

Let us resume the proof of the first part of Theorem~\ref{theorem-A}. Suppose that $\bigcap_{x \,\in\, \mathcal{C}\mathcal{I}(X,T)} \omega(x) = \emptyset$. By the compactness of $X$, there exist $x_{1},\cdots,x_{n} \in \mathcal{C}\mathcal{I}(X,T)$ such that
$$\bigcap_{j=1}^{n} \omega(x_{j}) \,=\, \emptyset.$$
Yet, one has
$$\emptyset \, \neq \, \bigcap_{x \,\in\, \mathcal{C}\mathcal{I}(X,T)} \overline{\mathcal{O}_T(x)}  \,\subseteq \,\bigcap_{j=1 }^{n} \overline{\mathcal{O}_T(x_{j})}$$
hence, using Lemma~\ref{lem_top}, we must have $\bigcap_{j=1 }^{n} \omega(x_{j}) \neq \emptyset$. This contradiction shows that $\bigcap_{x \,\in \,\mathcal{C}\mathcal{I}(X,T)} \omega(x) \neq \emptyset$.

\medskip

\noindent $(b)$ Take $x,y \in \mathcal{C}\mathcal{I}(X,T)$, and suppose, by contradiction, that there exist $\mu$ and $ \nu \in V_{T}(y)$ such that $\supp(\mu) \cap \omega(x) = \emptyset $ and $\supp(\nu) \cap \omega(x) = \emptyset $. Since $\mu$ and $\nu$ are  distinct invariant measures, there exists a continuous map $\varphi\colon X \to \mathbb{R}$ such that $\int \varphi \,d\mu \neq \int \varphi \,d\nu$. Thus, for some $\delta > 0$, one has
$$\supp(\mu) \cap \omega(x)^{\delta} \,= \,\emptyset \quad \quad \text{and} \quad \quad \supp(\nu) \cap \omega(x)^{\delta} \,=\, \emptyset.$$
Take a continuous function $\xi\colon X \to \mathbb{R}$ such that
$$\xi_{\mid_{\omega(x)^{\frac{\delta}{2}} }} \, = \,0 \quad \quad \text{ and } \quad \quad \xi_{\mid_{X \setminus \omega(x)^{\frac{3}{4}\delta} }} \, = \,1$$
and consider the continuous map $\psi\colon X \to \mathbb{R}$ defined by $\psi(x) = \varphi(x)\,\xi(x)$.

\begin{claim} $y \in \mathcal{I}(\psi,T)$.
\end{claim}

Indeed, one has
\begin{eqnarray*}
\int \psi \,d\mu  &=& \int_{\supp(\mu)} \psi \,d\mu  \,= \,\int_{\supp(\mu)} \varphi \,d\mu  \,=\, \int \varphi \,d\mu \\
&\neq& \int \varphi \,d\nu \,=\, \int_{\supp(\nu)} \varphi \,d\nu \,=\, \int_{\supp(\nu)} \psi \,d\nu  \,=\, \int \psi \,d\nu.
\end{eqnarray*}
Taking into account that $\mu, \nu \in V_{T}(y)$, we conclude that $y \in \mathcal{I}(\psi,T)$. 

We note that this claim ensures that $\psi \in \mathcal{H}(X,T)$. Therefore, $x \in \mathcal{I}(\psi,T)$ as well, since $x \in \mathcal{C}\mathcal{I}(X,T)$. Yet, $\psi_{\mid_{\omega(x)^{\delta/2} }} \, = \,0$, which implies that $x \notin \mathcal{I}(\psi,T)$. This contradiction indicates that we cannot have two such measures $\mu$ and $\nu$.

\medskip

\noindent $(c)$ Take $x,y \in \mathcal{C}\mathcal{I}(X,T)$ and $\mu, \nu \in V_{T}(y)$, and suppose, by contradiction, that
$$\supp(\mu) \cap  \Big(X \setminus \omega(x) \Big) \,  \neq \,\supp(\nu) \cap  \Big(X \setminus \omega(x) \Big).$$
Take $a \in X$ such that $a \in \supp(\mu) \cap  \left(X \setminus \omega(x) \right)$ and $a \notin \supp(\nu)$. Then there exists $\delta>0$ such that
$$B(a,2\delta) \cap \omega(x)^{\delta} \,=\, \emptyset \quad \quad \text{and} \quad \quad B(a,2\delta) \cap \supp(\nu)^{\delta} = \emptyset.$$
Since $\mu$ and $\nu$ are  distinct invariant measures, there exists a continuous map $\varphi\colon X \to \mathbb{R}$ such that $\varphi > 1$ and $\int \varphi d\mu \neq \int \varphi d\nu$. Take a continuous function $\xi\colon X \to [0,1]$ such that
$$\xi_{\mid_{ X \setminus \left(B(a,2\delta) \right) }} \, = \,0 \quad \quad \text{and} \quad \quad \xi_{\mid_{B(a,\delta) }} \, = \, 1 $$
and consider the continuous map $\psi\colon X \to \mathbb{R}$ defined by $\psi(x) = \varphi(x)\,\xi(x)$.

\begin{claim} $y \in \mathcal{I}(\psi,T)$.
\end{claim}

In fact, from
\begin{eqnarray*}
\int \psi \,d\nu  &=& \int_{\supp(\nu)} \psi \,d\nu \, =\, \int_{\supp(\nu)} \xi\, \varphi \,d\nu \,=\,0 \\
&<& \int_{B(a,\delta)} \varphi \, d\mu \,=\,  \int_{B(a,\delta)} \xi \,\varphi \, d\mu \, = \,\int_{B(a,\delta)} \psi \, d\mu \,\leq\, \int \psi \,d\mu
\end{eqnarray*}
we deduce that $\int \psi \,d\mu \neq  \int \psi \,d\nu$, and so $y \in \mathcal{I}(\psi,T)$ because $\mu, \nu \in V_{T}(y)$.

The claim yields that $\psi \in \mathcal{H}(X,T)$ and so $x \in \mathcal{I}(\psi,T)$ too, since $x \in \mathcal{C}\mathcal{I}(X,T)$. However, $\psi_{\mid_{\omega(x)^{\delta} }} \, = \,0$, which prevents $x$ to be in $\mathcal{I}(\psi,T)$. Thus the point $a$ cannot exist.

\medskip

\noindent $(d)$ Take $x_{0} \in \mathcal{C}\mathcal{I}(X,T)$. We are looking for a $T$-invariant set $\mathbb{B}_{x_0} \subseteq \overline{\mathcal{O}_T(x_0)}$ such that $\mathbb{B}_{x_0}$ is Baire generic in $\overline{\mathcal{O}_T(x_0)}$ and $x_0 \in \mathbb{B}_{x} \subseteq \mathcal{C}\mathcal{I}(X,T)$.

If  $\overline{\mathcal{O}_T(x_0)} = X$, take $\mathbb{B}_{x_0} = \mathcal{C}\mathcal{I}(X,T)$. Since $x_{0}$ is a point with dense orbit in $X$, and applying \cite[Corollary D]{CCSV21}, we deduce that $\mathcal{I}(T,\psi)$ is Baire generic in $X$ for every $\psi \in \mathcal{H}(X,T)$. Thus, $\mathcal{H}(X,T) = \mathcal{R}(X,T)$. Consequently, by \cite[Proposition 7.2]{CCSV21}, the set $\mathbb{B}_{x_{0}}$ is Baire generic in $X$.

If  $\overline{\mathcal{O}_T(x_0)} \neq X$, consider $Y =\overline{\mathcal{O}_T(x_{0})}$ and $S = T_{\mid_{Y}} \colon Y \to Y$.

\begin{claim} $x_{0} \in \mathcal{I}(S,\psi)$ for every $\psi \in \mathcal{H}(Y,S)$.
\end{claim}

Indeed, consider $\psi \in  \mathcal{H}(Y,S)$. There exist a continuous function $\varphi\colon X \to \mathbb{R}$ and $\delta >0$ such that
$$\varphi_{\mid_{Y}} \,=\, \psi \quad \quad \text{ and } \quad \quad \varphi_{\mid_{X \,\setminus \, Y^{\delta}}} \, = \,0.$$
As $\mathcal{I}(S,\psi) \neq \emptyset$, one has $\mathcal{I}(T,\varphi) \neq \emptyset$, and so $\varphi \in \mathcal{H}(X,T)$. Therefore, $x_{0} \in \mathcal{I}(T, \varphi)$, since $x_{0} \in \mathcal{C}\mathcal{I}(X,T)$. Moreover, as $x_{0} \in  Y = \overline{\mathcal{O}_T(x_{0})}$, we conclude that $x_{0} \in  \mathcal{I}(S,\psi)$.

Define $\mathbb{B}_{x_{0}} = \bigcap_{\psi \,\in\, \mathcal{H}(Y,S)} \mathcal{I}(S,\psi)$. By the the previous claim, we know that $x_{0} \in \mathbb{B}_{x_{0}}$. Combining this property with fact that $x_{0}$ is a point with dense orbit in $Y$, and applying \cite[Corollary D]{CCSV21}, we deduce that $\mathcal{I}(S,\psi)$ is Baire generic in $Y$ for every $\psi \in \mathcal{H}(Y,S)$. Thus, $\mathcal{H}(Y,S) = \mathcal{R}(Y,S)$. Consequently, by \cite[Proposition 7.2]{CCSV21}, the set $\mathbb{B}_{x_{0}}$ is Baire generic in $Y$.

We are left to show that $\mathbb{B}_{x_{0}} \subseteq  \mathcal{C}\mathcal{I}(X,T)$. Take $z \in \mathbb{B}_{x_{0}}$ and $\varphi \in \mathcal{H}(X,T)$. Since $\mathbb{B}_{x_{0}} \subseteq  \overline{\mathcal{O}_T(x_{0})}$, we know that $$\overline{\mathcal{O}_T(z)} \,\subseteq\,  \overline{\mathcal{O}_T(x_{0})} \, = \,Y.$$
Define $\psi = \varphi_{\mid_{Y}} \colon Y \to \mathbb{R}$. By assumption, $x_{0} \in \mathcal{C}\mathcal{I}(X,T)$, hence $x_{0} \in \mathcal{I}(T,\varphi)$. Thus, as $x_{0} \in Y=\overline{\mathcal{O}(x_{0})}$, one has $x_{0} \in  \mathcal{I}(S,\psi)$. Consequently, $\psi \in \mathcal{H}(Y,S)$, and so $z \in  \mathcal{I}(S,\psi)$ since $z \in \mathbb{B}_{x_{0}}$. Therefore, $z \in  \mathcal{I}(T,\varphi)$, confirming that $\mathbb{B}_{x_{0}} \subseteq  \mathcal{C}\mathcal{I}(X,T)$. The proof of Theorem~\ref{theorem-A} is complete.

\section{Proof of Theorem~\ref{theorem-B}}\label{se:compact}

Let $(X,d)$ be a compact metric space. Recall that, given $x \in X$,
$V_T(x) \subset\mathcal{P}_T(X)$ stands for the set of accumulation points, in the weak$^*$ topology, of the sequence of empirical measures $\big(\frac{1}{n}\,\sum_{j=0}^{n-1}\,\delta_{T^{j}x}\big)_{n \, \in \, \mathbb{N}}$.
It is known that $V_{T}(x)$ is connected and compact (cf. \cite[Proposition 3.8]{DGS76}). Therefore, given $\varphi \in C^0(X,\mathbb{R})$ and $x \in X$, the set $W_{\varphi}(x)$ is a compact interval of the real line. Indeed, the map
\begin{eqnarray*}
\mathcal{F}_{x,\,\varphi} \colon \quad V_{T}(x) & \to & W_{\varphi}(x)\\
\mu & \mapsto & \int \varphi \,d\mu
\end{eqnarray*}
is surjective and continuous, thus $W_{\varphi}(x)$ is a compact interval, namely
$$W_{\varphi}(x) \,=\, \Big[\liminf_{n\, \to \, +\infty}\, \frac{1}{n} \sum_{j=0}^{n-1} \,\varphi(T^{j}(x)),\,\,\limsup_{n\, \to \, +\infty}\, \frac{1}{n} \,\sum_{j=0}^{n-1} \,\varphi(T^{j}(x))\Big].$$
Clearly, $\# V_{T} (x) = 1$ if and only if $\# W_{\varphi}(x) = 1$ for every $\varphi \in C^0(X,\mathbb{R})$.

Let $\mathcal{D} = \{\varphi_{j}\colon j \in \mathbb{N} \cup \{0\}\}$ be a countable dense subset of $C^0(X,\mathbb{R})$ such that $\|\varphi_j\| > 0$ for every $j \in\mathbb{N} \cup \{0\}$. Given $\mu_{1}, \mu_{2} \in \mathcal{P}(X)$, define the metric $\rho$ on $\mathcal{P}(X)$ by
$$\rho(\mu_{1},\,\mu_{2})\, = \, \sum_{j=0}^{+\infty}\,\frac{\Big|\int \varphi_j\mathrm{d\mu_{1}}-\int \varphi_j\mathrm{d\mu_{2}}\Big|}{2^j\,\|\varphi_j\|}.$$
For each $j \in\mathbb{N} \cup \{0\}$, consider the map $f_j = \varphi_j/\|\varphi_j\|$. Then $\|f_j\|=1$ and
$$\rho(\mu_{1},\,\mu_{2}) = \sum_{j=0}^{+\infty}\,\frac{\Big|\int f_j\mathrm{d\mu_{1}}-\int f_j\mathrm{d\mu_{2}}\Big|}{2^j}.$$
Denote by $B_\rho(\nu, \varepsilon)$ the ball in the distance $\rho$ centered at $\nu$ with radius $\varepsilon$. The next result concerns the size of $V_{T} (x)$ and was inspired by \cite{HLT21v2}.

\begin{theorem}\label{theoremHLT_mod}
Let $(X,d)$ be a compact metric space without isolated points, $T\colon X \to X$ be a continuous function and $K$ be a nonempty subset (not necessarily compact) of $\mathcal{P}_{T}(X)$. Assume that there exists a dense subset $L$ of $K$ such that one of the following conditions holds:
\begin{itemize}
\item[$(a)$]  $\big\{x\in X\colon \,V_T(x) \cap B_\rho(\nu, \varepsilon) \neq \emptyset \big\}$ is dense in $X$ for every $\varepsilon >0$ and every $\nu \in L$.
\smallskip
\item[$(b)$]  $\big\{x\in X\colon \, L \subset V_T(x)\big\}$ is dense in $X$.
\end{itemize}
Then $\big\{x\in X\colon \, K \subset V_T(x)\big\}$ is Baire generic in $X$.
\end{theorem}

\begin{proof}\label{se:proof-HLT}

Let $(X,d)$ be a compact metric space without isolated points, $T\colon X \to X$ be a continuous function and $K$ be a nonempty subset (not necessarily compact) of $\mathcal{P}_{T}(X)$. Note that $V_{T}(x) = V_{T}(T^{n}x)$ for every $n \in \mathbb{N}$ and $x \in X$. By \cite[Lemma 2.2]{HLT21v2}, one has
$$\big\{x\in X \colon \, K \subset V_T(x)\big\} \,=\, \{x\in X\colon \, \overline{K} \subset V_T(x)\big\}.$$
So, without loss of generality, we may assume that $K$ is a nonempty compact subset of $\mathcal{P}_{T}(X)$. Therefore, there exist sequences $\{U_i\}_{i\,\in\,\mathbb{N}}$ and $\{V_i\}_{i\,\in\,\mathbb{N}}$ of open balls in $\mathcal{P}(X)$, with respect to the metric $\rho$, such that:
\begin{itemize}
\item $V_i \subset \overline{V}_i \subset U_i$;
\smallskip
\item $\lim_{i \, \to \, +\infty} \,\diam(U_i) \,=\, 0$;
\smallskip
\item  $V_i\cap K \neq \emptyset$;
\smallskip
\item Each point of $K$ lies in infinitely many elements of the sequence $\{V_i\}_{i\,\in\,\mathbb{N}}$.
\end{itemize}

Define $\mathcal{V}(U_i) = \big\{x\in X\colon \, V_T(x) \cap U_i \neq \emptyset\big\}$. Then
$$\bigcap_{i=1}^{+\infty}\,\mathcal{V}(U_i) \,=\,  \big\{x\in X\colon \, K \subset V_T(x)\big\}$$
and
$$\bigcap_{N=1}^{+\infty} \,\bigcup_{n\,>\,N} \, \Big\{x \in X \colon \,\, \frac{1}{n} \sum_{j=0}^{n-1} \,\delta_{T^{j} x} \in V_{i}\Big\} \,\subset \, \mathcal{V}\left(U_{i}\right).$$
For each $N \in \mathbb{N}$ and $i \in \mathbb{N}$, set
$$U(N,i) \, =\, \bigcup_{n \,> \,N} \, \Big\{x \in X \colon \, \frac{1}{n} \sum_{j=0}^{n-1} \,\delta_{T^{j} x} \in V_{i}\Big\}.$$
Then $U(N,i)$ is an open subset of $X$. Therefore, if under each of the conditions $(a)$ or $(b)$ of the statement of Theorem~\ref{theoremHLT_mod} we guarantee that $U(N,i)$ is dense in $X$ for every $N \in\mathbb{N}$ and $i\in\mathbb{N}$, then we are sure that the set $\big\{x\in X \colon \, K \subset V_T(x)\big\}$ is Baire generic in $X$.

\medskip

\noindent $(a)$ Assume that there exists a dense subset $L$ of $K$ such that $\big\{x\in X\colon \,V_T(x) \cap B_\rho(\nu, \varepsilon) \neq \emptyset \big\}$ is dense in $X$ for every $\varepsilon >0$ and every $\nu \in L$. Fix $N, i \in \mathbb{N}$. As $L$ is dense in $K$, $K \cap V_{i} \neq \emptyset$ and this is an open subset of $K$, then we may find $\nu \in L$ such that $\nu \in K \cap V_{i}$. Choose $\varepsilon>0$ such that $B_\rho(\nu, \varepsilon) \subseteq V_{i}$. Then
$$\big\{x\in X\colon \, V_T(x) \cap B_\rho(\nu, \varepsilon) \neq \emptyset \big\} \,\subseteq\,  U(N,i)$$
and so $U(N,i)$ is dense in $X$.

\medskip

\noindent $(b)$  Suppose now that $\big\{x\in X\colon \, L \subset V_T(x)\big\}$ is dense in $X$. Given $N, i \in \mathbb{N}$, as $L$ is dense in $K$, $K \cap V_{i} \neq \emptyset$ and this is an open set of $K$, then there exists $\nu \in L$ such that $\nu \in K \cap V_{i}$. Choose $\varepsilon>0$ such that $B_\rho(\nu, \varepsilon) \subseteq V_{i}$. Then
$$\big\{x\in X\colon \, L \subset V_T(x)\big\} \,\subseteq\, \big\{x\in X\colon \, \nu \in V_T(x)\big\}$$
and so the latter set is also dense in $X$. Taking into account that
$$\big\{x\in X\colon \, \nu \in  V_T(x)\big \}\,\subseteq\,  U(N,i)$$
we again conclude that $U(N,i)$ is dense in $X$.

\end{proof}

Using Theorem~\ref{theoremHLT_mod}, we generalize \cite[Theorem D(3)]{HLT21v2} as follows.

\begin{corollary}\label{cor:ergodic}
Let $(X,d)$ be a compact metric space without isolated points and $T\colon X \to X$ be a continuous map. Then:
\begin{itemize}
\item[$(a)$] If $K$ is a nonempty subset (not necessarily compact) of $\mathcal{P}_{T}(X)$ and
$$\big\{x\in X\colon \, K \subset V_T(x)\big\} \, \neq \, \emptyset$$
then for every $x_{0} \in \big\{x\in X\colon K \subset V_T(x)\big\}$ one has
$$\overline{\bigcup_{\mu \,\in\, K} \supp \mu} \,\subseteq\, \omega(x_{0}) \,\subseteq\, \overline{\big\{x\in X\colon\, K \subset V_T(x)\big\}}.$$
\item[$(b)$] If $\mathcal{M}(X,T) = X$ and $\big\{x\in X\colon\, V_T(x) = \mathcal{P}_{T}(X)\big\}$ is nonempty, then $T$ is transitive and $\big\{x\in X\colon\, V_T(x) = \mathcal{P}_{T}(X)\big\}$ is Baire generic in $X$.
\medskip
\item[$(c)$] If $T$ is not uniquely ergodic and $\big\{x\in X\colon \, V_T(x) = \mathcal{P}_{T}(X)\big\}$ is Baire generic in $X$, then $\bigcap_{\varphi \,\in \,\mathcal{H}(X,T)} \mathcal{I}(T,\varphi)$ is Baire generic as well.
\medskip
\item[(d)] If $\mathcal{M}(X,T) = X$ and $\mathcal{P}_{T}^{erg}(X)$ is dense in $\mathcal{P}_{T}(X)$, then
\smallskip
\begin{itemize}
\item[(i)] $\big\{x\in X\colon \, V_T(x) = \mathcal{P}_{T}(X)\big\}$ is Baire generic in $X$.
\smallskip
\item[(ii)] $T$ is transitive.
\end{itemize}
\end{itemize}
\end{corollary}

\begin{proof}\label{se:proof-corB}

\noindent $(a)$ Let $K$ be a nonempty subset (not necessarily compact) of $\mathcal{P}_{T}(X)$ such that
$$\big\{x\in X\colon \, K \subset V_T(x)\big\} \, \neq \, \emptyset.$$
Consider $x_{0} \in \big\{x\in X\colon K \subset V_T(x)\big\}$ and $\mu \in K$. Then $\supp (\mu) \subseteq \omega(x_{0})$, and so
$$\overline{\bigcup_{\mu\, \in \, K} \supp \mu} \,\subseteq\, \omega (x_{0}).$$
Moreover,
$$\big\{T^{j}(x_{0})\colon \,j \in \mathbb{N} \cup \{0\} \big\} \,\subseteq\, \big\{x\in X\colon \, K \subseteq V_T(x)\big\}.$$
This implies that
$$\overline{\bigcup_{\mu\, \in \,K}\, \supp \mu} \,\subseteq \,\omega (x_{0}) \,\subseteq \,\overline{\big\{x\in X\colon K \subseteq V_T(x)\big\}}.$$

\medskip

\noindent $(b)$ Suppose that $\mathcal{M}(X,T) = X$ and $\big\{x\in X\colon \, V_T(x) = \mathcal{P}_{T}(X)\big\}\neq \emptyset$. Let $x_{0}$ be a point of $X$ such that $V_T(x_{0}) = \mathcal{P}_{T}(X)$. By item $(a)$ with $K = \mathcal{P}_{T}(X)$,
$$X \,=\, \mathcal{M}(X,T) \,=\,\overline{\bigcup_{\mu\, \in \,\mathcal{P}_{T}(X)} \,\supp \mu}\,\subseteq\, \omega (x_{0}) \, \subseteq\, \overline{\big\{x\in X \colon \, V_T(x) = \mathcal{P}_{T}(X) \big\}}.$$
These inclusions imply that $x_{0}$ has a dense orbit. Thus $T$ is transitive, since $X$ has no isolated points. Moreover, $\big\{x\in X\colon \,V_T(x) = \mathcal{P}_{T}(X) \big\}$ is dense in $X$. Applying item $(b)$ of Theorem~\ref{theoremHLT_mod}, we conclude that $\big\{x\in X \colon \, V_T(x) =  \mathcal{P}_{T}(X)\}$ is Baire generic in $X.$

\medskip

\noindent $(c)$ Suppose that $T$ is not uniquely ergodic and $\big\{x\in X\colon \, V_T(x) = \mathcal{P}_{T}(X)\big\}$ is Baire generic in
$X$. These assumptions ensure that $\mathcal{H}(X,T)$ is not empty. Using \cite[Lemma 4.4]{Tian}, we conclude that
$$\big\{x\in X\colon \, V_T(x) = \mathcal{P}_{T}(X)\big\}\, \subseteq\,  \bigcap_{\varphi \,\in \,\mathcal{H}(X,T)} \,\mathcal{I}(T,\varphi)$$
so the latter set is Baire generic as well.

\medskip

\noindent $(d)-(i)$ Assume that $\mathcal{M}(X,T) = X$ and $\mathcal{P}_{T}^{erg}(X)$ is dense in $\mathcal{P}_{T}(X)$. From \cite[Proposition 5.7]{DGS76}, we know that $\mathcal{P}_{T}^{erg}(X)$ is Baire generic in $\mathcal{P}_{T}(X)$. Moreover, by \cite[Lemma 5.1]{KOR2016}, there exists an invariant measure $\mu_0$ such that $\supp \mu_0 = \mathcal{M}(X,T) = X$. Using \cite[Proposition 21.11]{DGS76}, we deduce that the set $\big\{\eta \in \mathcal{P}_{T}(X)\colon \, \supp \eta = X\big\}$ is Baire generic in $\mathcal{P}_{T}(X)$. Therefore, the set $\big\{\eta\in \mathcal{P}_{T}(X) \colon \, \supp \eta = X\big\} \cap \mathcal{P}_{T}^{erg}(X)$ is also Baire generic in $\mathcal{P}_{T}(X)$.

Given $\mu \,\in \,\mathcal{P}_{T}(X)$, denote by
$\mathcal{B}(\mu)$ the basin of attraction of $\mu$, defined by
\begin{equation*}\label{def:basin}
\mathcal{B}(\mu) \, = \, \Big\{x \in X \colon \lim_{n \,\to\, +\infty} \,\frac{1}{n}\,\sum_{j=0}^{n-1} \psi(T^{j}(x)) = \int \psi \,d \mu, \,\, \forall\, \psi \in C^0(X,\mathbb{R})\Big\}.
\end{equation*}
Recall that, if $\eta$ is a $T$-invariant ergodic probability measure such that $\supp \eta = X$, then its basin of attraction is dense in $X$. Thus, if we take the Baire generic set
$$L \,=\, \big\{\eta \in \mathcal{P}_{T}(X) \colon \, \supp \eta = X\big\} \cap \mathcal{P}_{T}^{erg}(X)$$
then, by item $(a)$ of Theorem \ref{theoremHLT_mod}, we conclude that $\big\{x\in X \colon\, V_T(x)= \mathcal{P}_{T}(X)\big\}$ is Baire generic in $X.$

\medskip

\noindent $(d)-(ii)$  We now apply item $(b)$ to infer that $T$ is transitive. The proof of Corollary~\ref{cor:ergodic} is complete.

\begin{remark}
It was already known that, if $\mathcal{P}_{T}^{erg}(X)$ is dense in $\mathcal{P}_{T}(X)$, then the restriction of $T$ to $\mathcal{M}(X,T)$ is transitive (cf. \cite[Proposition 7.2]{KK}).
\end{remark}

\end{proof}

We note that Theorem~\ref{theorem-B} is now a direct consequence of the previous items $(c)$ and $(d)$. 

\subsection{Sensitivity to initial conditions}

In \cite[Theorem 2.2]{CCSV21}, it was proved that if $X$ is $(T,\varphi)$-sensitive (for the precise notion, see Definition~\ref{defTphisensitive} in Section~\ref{se:chaotic}) then either $T$ has sensitivity to initial conditions or $\mathcal{I}(T,\varphi)$ has nonempty interior.
The next consequence of Theorem~\ref{theoremHLT_mod} provides a setting where the previous dichotomy is untied.

\begin{corollary}\label{maincorollary_sensensitive} Let $X$ be a compact metric space without isolated points, $T\colon X \to X$ be a continuous map whose measure center coincides with $X$ and $\varphi \in C^0(X,\mathbb R)$. If $X$ is $(T,\varphi)$-sensitive, then $T$ has sensitivity to initial conditions.
\end{corollary}

\begin{proof}\label{se:proof-corA}

Let $X$ be a compact metric space without isolated points and $T\colon X \to X$ be a continuous map.

\begin{lemma}\label{lemma_dense}
$\mathcal{M}(X,T) \,=\, \overline{\bigcup_{\eta\, \in\, \mathcal{P}_{T}^{erg}(X) } \mathcal{B}(\eta) \cap \supp(\eta)}$.
\end{lemma}

\begin{proof}
Clearly
$$\overline{\bigcup_{\eta\, \in \,\mathcal{P}_{T}^{erg}(X)} \mathcal{B}(\eta) \cap \supp(\eta)} \,\,\subseteq\,\, \overline{\bigcup_{\eta \,\in\, \mathcal{P}_{T}^{erg}(X)} \supp(\eta)} \,\,\subseteq\,\, \mathcal{M}(X,T).$$
We are left to show the converse inclusion. According to \cite[Lemma 5.1]{KOR2016}, there exists an invariant measure $\mu_0$ such that $\supp\mu_0 = \mathcal{M}(X,T)$. Take $z$ in $\mathcal{M}(X,T)$ and let $U$ be an open neighborhood of $z$ in $X$. Then $\mu_0(U) >0 $ and, by the ergodic decomposition of $\mu_0$, there exists a $T$-invariant ergodic measure $\eta_0$ such that $\eta_0(U)> 0 $. Since
$\eta_0\big(\mathcal{B}(\eta_0)\cap \supp (\eta_0)\big) = 1$, one must have $\mathcal{B}(\eta_0) \cap \supp (\eta_0) \cap U \neq \emptyset$. As $U$ is arbitrary, we conclude that $z$ belongs to $\overline{\bigcup_{\eta\, \in\, \mathcal{P}_{T}^{erg}(X) } \mathcal{B}(\eta) \cap \supp(\eta)}$.
\end{proof}

Assume that $X$ is $(T,\varphi)$-sensitive, for some $\varphi \in C^0(X, \mathbb{R})$, and that $\mathcal{M}(X,T)= X$. By
Lemma~\ref{lemma_dense}, one has
$$\overline{\bigcap_{\psi\, \in \,C^0(X,\mathbb{R})} \,X \setminus {\mathcal I}(T,\psi)} \, \,= \,\, X$$
since
$$\bigcup_{\eta\, \in\, \mathcal{P}_{T}^{erg}(X)} \,\mathcal{B}(\eta) \cap \supp(\eta)  \,\,\subseteq \,\bigcap_{\psi\, \in \,C^0(X,\mathbb{R})} \,X \setminus {\mathcal I}(T,\psi).$$
\smallskip

\noindent Therefore, $X \setminus {\mathcal I}(T,\varphi)$ is dense in $X$. So, as $X$ is $(T,\varphi)$-sensitive, we may apply \cite[Theorem~2.3]{CCSV21} and thereby conclude that $T$ has sensitivity to initial conditions.
\end{proof}

\section{Proof of Theorem~\ref{theorem-C}}\label{se:proof-C}

Denote by $\Omega(T)$ the set of non-wandering points of $T$. Since $T$ is a transitive continuous map and $X$ has no isolated points, one has $\Omega(T) = X$ (see \cite[Theorem~1.4]{AC12}). Moreover, the set of minimal points of $T$ is dense in $\Omega(T)$ because $T$ satisfies the shadowing property (cf. \cite[Corollary 1]{Moo11}). Thus, the set $\AP(X,T)$ of minimal points is dense in $X$. As $\AP(X,T) \subseteq \mathcal{M}(X,T)$ and $\mathcal{M}(X,T)$ is closed, one has $\mathcal{M}(X,T) = X$. We note that, from the last equality and Lemma~\ref{lemma_dense}, we deduce that $\bigcup_{\mu \,\in\, \mathcal{P}_{T}^{erg}(X)} \mathcal{B}(\mu)$ is dense in $X$.

Taking into account that $T$ is transitive and has the shadowing property, then one has $\overline{\mathcal{P}_{T}^{erg}(X)} = \mathcal{P}_{T}(X)$ (cf. \cite[Theorem A]{LO18}). Consequently, since $\mathcal{M}(X,T) = X$ and $T$ is not uniquely ergodic, Corollary \ref{cor:ergodic} 
implies that  $\mathcal{H}(X,T) \neq \emptyset$ and  $\bigcap_{\varphi \,\in \,\mathcal{H}(X,T)} \mathcal{I}(T,\varphi)$ is Baire generic in $X$.


From \cite[Theorem 1.5]{DOT18} we know that, for any continuous map on a compact metric space satisfying the shadowing property, one has $\mathcal{H}(X,T) \neq \emptyset$ if and only if $T$ has positive topological entropy. So, $T$ has positive topological entropy.

As $\mathcal{C}\mathcal{I}(X,T)$ is Baire generic in $X$, there exists $\varphi \in C^0(X, \mathbb{R})$ such that $\mathcal{I}(T,\varphi)$ is also Baire generic.
Due to the fact that $T$ is transitive, this implies that $X$ is $(T,\varphi)$-sensitive by  \cite[Corollary D]{CCSV21}. Using Corollary~\ref{maincorollary_sensensitive},
we finally conclude that $X$ has sensitivity to initial conditions. This completes the proof of Theorem~\ref{theorem-C}.

\subsection{On the condition $\mathcal{M}(X,T) = \Omega(T)$}

In the proof of Theorem~\ref{theorem-C}, it was essential to know that, under its assumptions, the set of the nonwandering points of $T$ coincides with its measure center. Actually, this property holds $C^{1}$-generically. Let $X$ be a compact smooth connected manifold without boundary of dimension $\mathrm{dim}\, X \geq 2$. Denote by $\mathrm{Diff}^1(X)$ the space of $C^{1}$ diffeomorphisms of $X$ endowed with the usual $C^1$ topology.

\begin{proposition}\label{theorem.generic.periodic.dense}
There exists a Baire generic subset $\mathfrak{R} \subseteq \mathrm{Diff}^1(X)$ such that, for every $T \in \mathfrak{R}$, one has
$$\overline{\mathrm{Per}(X,T)} \, = \, \Omega(T) \, = \, \mathcal{M}(X,T).$$
\end{proposition}

\begin{proof}
Recall that every periodic point of $T$ is an element of the measure center $\mathcal{M}(X,T)$. Moreover:

\begin{lemma}\label{lemma.center.measure.nonwandering}
Let $(X,d)$ be a compact metric space and $T\colon X \to X$ be a continuous function. Then $\mathcal{M}(X,T) \subseteq \Omega(T)$.
\end{lemma}

\begin{proof}
By \cite[Lemma 5.1]{KOR2016}, there exists a $T$-invariant probability measure $\mu$ such that $\supp \mu = \mathcal{M}(X,T)$. Take $z \in \mathcal{M}(X,T) = \supp \mu$ and let $U$ be an open neighborhood of $z$. So, $\mu (U)>0$ and, by Poincar\'e recurrence  theorem, $\mu$-almost every point in $U$ is recurrent. Thus, there exists $x \in U$	such that $T^{N}(x) \in U$ for some $N \in \mathbb{N}$. Consequently, $T^{N}(U)\cap U \neq \emptyset$ and $z \in \Omega(T)$.
\end{proof}

Let us resume the proof of Proposition~\ref{theorem.generic.periodic.dense}. From Pugh's Closing Lemma \cite{Pugh}, there exists a Baire generic subset $\mathfrak{R} \subseteq \mathrm{Diff}^1(X)$ such that, for every $T \in \mathfrak{R}$, one has $\Omega(T) = \overline{\mathrm{Per}(X,T)}$. So, applying Lemma~\ref{lemma.center.measure.nonwandering}, we conclude that if $T \in \mathfrak{R}$ then $\Omega(T) = \mathcal{M}(X,T)$, since
$$\Omega(T) \,=\, \overline{\mathrm{Per}(X,T)} \,\subseteq \,\mathcal{M}(X,T) \,\subseteq\, \Omega(T).$$

\end{proof}

For instance, if $T$ is Anosov (or, more generally, $C^1$ structurally stable), then $\Omega(T) = \mathcal{M}(X,T)$. If, in addition, the space $\mathcal{P}_T(X)$ contains a measure with full support, then $\Omega(T) = X = \mathcal{M}(X,T)$.

\begin{remark} Clearly, $\mathcal{M}(X,T) = X$ if and only if the space $\mathcal{P}_T(X)$ contains a measure with full support. Indeed, if the space $\mathcal{P}_T(X)$ contains a measure with full support then, by definition, $\mathcal{M}(X,T) = X$.  Conversely, assume that $\mathcal{M}(X,T) = X$. Then, by \cite[Lemma 5.1]{KOR2016}, there exists a $T$-invariant probability measure $\mu$ such that $\supp \mu = \mathcal{M}(X,T)$. There may not exist such a measure $\mu$ ergodic, though, as an example in \cite{Weiss} confirms.
\end{remark}

\begin{remark}
It was already known that the set of irregular points is Baire generic in $X$ for $C^1$-generic diffeomorphisms in $\mathrm{Diff}^1(X)$ (cf. \cite[Theorem 3.14]{ABC11}).
\end{remark}

\begin{remark}
A statement similar to the one in Proposition~\ref{theorem.generic.periodic.dense} is valid for the space of homeomorphisms $\Hom(X,\mu_0)$ we considered in Example~\ref{ex:3}. Indeed, if $f \in \Hom(X,\mu_0)$ then $\Omega(f)= X$, due to item $(ii)$ of the definition of Lebesgue-like map $\mu_0$; and it was proved in \cite{DaaFokk} that, for a generic map $f$ in $\Hom(X,\mu_0)$, the set $\mathrm{Per}(f)$ is dense in $X$.
\end{remark}

\subsection{On the condition $\mathcal{M}(X,T) = \overline{\mathrm{Per}(X,T)}$}

In \cite{PfSu}, we find an example of a continuous map $T\colon X \to X$ acting on a compact metric space whose set $\mathcal{P}_T^{erg}(X)$ is dense in $\mathcal{P}_T(X)$, though $T$ has no periodic points. This indicates that the denseness of ergodic measures in $\mathcal{P}_T(X)$ is not enough to ensure that $\mathrm{Per}(X,T)$ is dense in $\mathcal{M}(X,T)$. A sufficient condition is addressed in the next proposition.

\begin{proposition}\label{prop.periodic.dense}
Let $(X,d)$ be a compact metric space without isolated points and $T\colon X \to X$ be a continuous map. If the space of $T$-invariant Borel probability measures is the closed convex hull of the set of ergodic measures supported on periodic orbits of $T$, then $\mathrm{Per}(X,T)$ is dense in $\mathcal{M}(X,T)$.
\end{proposition}

\begin{proof} (Inspired by the proof of \cite[Theorem 30]{KLO16}) Let $U$ be a nonempty open subset of $X$ such that $U \cap \mathcal{M}(X,T) \neq \emptyset$. By \cite[Lemma 5.1]{KOR2016}, there exists a $T$-invariant probability measure $\mu$ such that
$\supp \mu = \mathcal{M}(X,T)$. Thus, $\mu (U) > 0$. Let
$(\nu_{n})_{n \, \in \,\mathbb{N}}$ be a sequence of probability measures which weak$^*$-converges to $\mu$, where $\nu_{n}$ is a finite convex sum of ergodic measures supported on periodic orbits of $T$. Since
$$\liminf_{n \,\to \,+\infty} \,\nu_{n}(U) \,\geq \,\mu(U) \, > \, 0$$
there surely exist an ergodic measure $\eta$ supported on a periodic orbit such that $\eta (U) > 0$. Thus, $U \cap \mathrm{Per}(X,T) \neq \emptyset$.
\end{proof}

\subsection{On the joint assumption of transitivity and shadowing}

A comment is due regarding the relation between the two main assumptions of Theorem~\ref{theorem-C}. The next result is inspired by \cite[Theorem 30]{KOR2016} and provides a necessary and sufficient condition for transitivity within systems with the shadowing property.

\begin{proposition}\label{sha_trans_equi}
Let $(X,d)$ be a compact metric space without isolated points and $T\colon X \to X$ be a continuous map satisfying the shadowing property. Then the following assertions are equivalent:
\begin{enumerate}
\item[$(a)$] $T$ is transitive.
\smallskip

\item[$(b)$] $\mathcal{P}_{T}^{erg}(X)$ is dense in $\mathcal{P}_{T}(X)$.
\end{enumerate}
\end{proposition}

\begin{proof} \noindent $(a) \Rightarrow (b)$ Suppose that $T$ is transitive and has the shadowing property. Then, by \cite[Theorem A]{LO18}), one has $\overline{\mathcal{P}_{T}^{erg}(X)} = \mathcal{P}_{T}(X)$.
\smallskip

\noindent $(b) \Rightarrow (a)$ Assume that $T\colon  X \to X$ has the shadowing property and $\overline{\mathcal{P}_{T}^{erg}(X)} = \mathcal{P}_{T}(X)$. By  Lemma~\ref{lemma.center.measure.nonwandering}, we know that $\mathcal{M}(X,T) \subseteq \Omega(T)$. Moreover, the set $\AP(X,T)$ of minimal points of $T$ is dense in $\Omega(T)$ because $T$ satisfies the shadowing property (cf. \cite[Corollary 1]{Moo11}). Since $\AP(X,T) \subseteq \mathcal{M}(X,T)$ and $\mathcal{M}(X,T)$ is closed, one concludes that $\mathcal{M}(X,T) = \Omega(T)$.

\begin{lemma}\label{measure_center_transitive}
Let $T\colon X \to X$ be a continuous map on a compact metric space $(X,d)$ without isolated points and $C_R(X,T)$ denote its chain recurrent set. If $\mathcal{M}(X,T) = C_R(X,T)$ and $\mathcal{P}_{T}^{erg}(X)$ is dense in $\mathcal{P}_{T}(X)$, then $T$ is transitive.
\end{lemma}

\begin{proof}
Since $\overline{\mathcal{P}_{T}^{erg}(X)} = \mathcal{P}_{T}(X)$, the map $T|_{\mathcal{M}(X,T)}\colon \mathcal{M}(X,T) \to \mathcal{M}(X,T)$ is transitive (cf. \cite[Proposition 6.4]{KK}). Hence $T|_{C_R(X,T)}\colon C_R(X,T) \to C_R(X,T)$ is transitive due to the assumption $\mathcal{M}(X,T) = C_R(X,T)$. Therefore, $T$ is transitive in $X$ (cf. \cite[Theorem 1.1]{MM00}).
\end{proof}

We now complete the proof of $(b) \Rightarrow (a)$. Since $T$ satisfies the shadowing property, one has $C_R(X,T) = \Omega(T)$ (see \cite[Theorem 3.1.2]{AH94}). So, $\mathcal{M}(X,T) = C_R(X,T)$. As, by assumption, $\mathcal{P}_{T}^{erg}(X)$ is dense in $\mathcal{P}_{T}(X)$, by applying Lemma~\ref{measure_center_transitive} we conclude that $T$ is transitive.

\end{proof}

\section{The orbits of completely irregular points}\label{se:ci-orbits}

We recall from (cf. \cite[Theorem 4.3 and Theorem 4.5]{Tian}) that, if $T$ is not uniquely ergodic, satisfies the almost specification property and $\mathcal{M}(X,T) = X$, then
$$\bigcap_{x \,\in\, \mathcal{C}\mathcal{I}(X,T)} \omega(x) \, = \, X.$$
In particular, under these assumptions every point in $\mathcal{C}\mathcal{I}(X,T)$ has a dense orbit by $T$. As an application of Theorem \ref{theorem-A}, we extend this result to transitive continuous maps with the shadowing property.

\begin{proposition}\label{theo_tra_sha}
Let $(X,d)$ be a
compact metric space without isolated points and consider a transitive continuous map $T\colon X \to X$ with the shadowing property. If $T$ is not uniquely ergodic, then
$$\mathcal{C}\mathcal{I}(X,T) \,\subseteq \,\mathrm{Trans}(X,T).$$
\end{proposition}

\begin{proof}

If $T$ is minimal, there is nothing else to prove. Suppose that $T$ is not minimal. Since $T$ is a transitive continuous map with the shadowing property, one has that $\Omega(T) = X$ and,  by Proposition~\ref{sha_trans_equi}, the set $\mathcal{P}_{T}^{erg}(X)$ is dense in $\mathcal{P}_{T}(X)$. Moreover, the set $\AP(X,T)$ of minimal points of $T$ is dense in $\Omega(T)$ because $T$ satisfies the shadowing property (cf. \cite[Corollary 1]{Moo11}). Since $\AP(X,T) \subseteq \mathcal{M}(X,T)$ and $\mathcal{M}(X,T)$ is closed, one concludes that $\mathcal{M}(X,T) = \Omega(T) = X$. Bringing together the properties $\mathcal{M}(X,T) = X$ and $\overline{\mathcal{P}_{T}^{erg}(X)} = \mathcal{P}_{T}(X)$, Corollary~\ref{cor:ergodic} yields that $\big\{x\in X\colon \, V_T(x) = \mathcal{P}_{T}(X)\big\}$ is Baire generic in $X$.

Take $z_0 \in \mathcal{C}\mathcal{I}(X,T)$ and $x_0 \in  \big\{x \in X\colon \, V_T(x) = \mathcal{P}_{T}(X)\big\}$, whose existence is guaranteed by Corollary~\ref{cor:ergodic}. We observe that both $z_0$ and $x_0$ are in $\mathcal{C}\mathcal{I}(X,T)$. Therefore, by item $(c)$ of Theorem \ref{theorem-A}, for every
$\mu, \nu \in V_{T}(x_0) = \mathcal{P}_{T}(X)$, one has
$$\supp(\mu) \cap  \big(X \setminus \omega(z_0)\big) \, = \,\supp(\nu) \cap  \big(X \setminus \omega(z_0)\big).$$

\begin{lemma}\label{lem_AP}
If $T$ is not minimal, there are $p$ and $q \in \AP(X,T)$ such that $\overline{\mathcal{O}(p)} \,\cap \,\overline{\mathcal{O}(q)} = \emptyset.$
\end{lemma}

\begin{proof}
Suppose, by contradiction, that for all $p,q \in \AP(X,T)$ one has $\overline{\mathcal{O}(p)} \,\cap\, \overline{\mathcal{O}(q)} \neq \emptyset$. Then
$$\overline{\mathcal{O}(p)}\,=\,\overline{\mathcal{O}(p)} \cap \overline{\mathcal{O}(q)} \,\subseteq\, \overline{\mathcal{O}(q)} \,=\,\overline{\mathcal{O}(p)} \cap \overline{\mathcal{O}(q)} \,\subseteq\, \overline{\mathcal{O}(p)}$$
because $p$ and $q$ are minimal points. Consequently, $\overline{\mathcal{O}(p)} = \overline{\mathcal{O}(q)}$ for every $p,q \in \AP(X,T)$.

\begin{claim} There exists  $b \in \AP(X,T)$ such that $\overline{\mathcal{O}(b)} \neq X$.
\end{claim}

Otherwise, if for every $b \in \AP(X,T)$ one has $\overline{\mathcal{O}(b)} = X$, then we may take an arbitrary point $x \in X$ and, by Birkhoff's theorem \cite{B66}, we find a minimal point $c$ such that $\overline{\mathcal{O}(c)} \subseteq \overline{\mathcal{O}(x)}$. So, 
$$X \,=\,\overline{\mathcal{O}(c)} \,\subseteq \,\overline{\mathcal{O}(x)}.$$
Thus, $x$ has a dense orbit. This shows that $T$ is minimal, contradicting the assumption. The claim is proved.

\medskip

Consider $b \in \AP(X,T)$ such that $\overline{\mathcal{O}(b)} \neq X$. Then $X \setminus \overline{\mathcal{O}(b)} \neq \emptyset$ is an open non-empty set. Take $v \in X \setminus \overline{\mathcal{O}(b)}$ with $d(v,\overline{\mathcal{O}(b)})=\varepsilon >0$. As $\AP(X,T)$ is dense in $\Omega(T) = X$, there exists $w \in \AP(X,T)$ such that $d(v,w) < \varepsilon$. Yet, this is a contradiction, since $\overline{\mathcal{O}(p)} = \overline{\mathcal{O}(q)}$ for every $p,q \in \AP(X,T)$ and so
$$\varepsilon \,< \,d(v,\overline{\mathcal{O}(b)})\, = \,d(v,\overline{\mathcal{O}(w)})\,<\, \varepsilon.$$
Therefore, there must exist $p$ and $q \in \AP(X,T)$ such that $\overline{\mathcal{O}(p)} \cap \overline{\mathcal{O}(q)} = \emptyset$. This completes the proof of the lemma.

\end{proof}

Let us resume the proof of Proposition~\ref{theo_tra_sha}. Since $T$ is not minimal, take $p, q \in \AP(X,T)$ such that $\overline{\mathcal{O}(p)} \cap \overline{\mathcal{O}(q)} = \emptyset$, provided by Lemma \ref{lem_AP}. Consider $\eta_{p}$ and $\eta_{q}$  ergodic probability measures supported on $\overline{\mathcal{O}(p)}$ and $\overline{\mathcal{O}(q)}$, respectively. Thus, $\supp(\eta_{p})  \cap  \supp(\eta_{q})  = \emptyset$, and therefore
$$\supp(\eta_{p}) \cap  \big(X \setminus \omega(z_0)\big) \, = \,\supp(\eta_{q}) \cap  \big(X \setminus \omega(z_0)\big) = \emptyset.$$
This implies (see item (c) of Theorem~\ref{theorem-A}) that
$$\supp(\mu) \cap  \big(X \setminus \omega(z_0)\big) \, = \,\supp(\nu) \cap  \big(X \setminus \omega(z_0)\big) = \emptyset \quad \quad \forall\,\mu, \nu \in V_{T}(x_0) = \mathcal{P}_{T}(X).$$
Thus
$$\supp(\mu) \subseteq \omega(z_0) \quad \quad \forall\, \mu \in V_{T}(x_0) = \mathcal{P}_{T}(X).$$

\noindent Hence, $ X = \mathcal{M}(X,T) \subseteq \omega(z_0)$. As $z_0$ is arbitrary in $\mathcal{C}\mathcal{I}(X,T)$, we conclude that
$$X \,=\, \mathcal{M}(X,T) \, =\, \bigcap_{x \,\in\, \mathcal{C}\mathcal{I}(X,T)} \omega(x).$$
\end{proof}

As a matter of fact, the proof of Proposition~\ref{theo_tra_sha} also shows that:

\begin{proposition}\label{theo_tra_sha_mod}
Let $(X,d)$ be a compact metric space without isolated points and consider a continuous map $T\colon X \to X$. If $T$ is not uniquely ergodic, $\big\{x\in X\colon\, V_T(x) = \mathcal{P}_{T}(X)\big\}$ is nonempty and the set $\AP(X,T)$ of minimal points is dense in $X$, then
$$\mathcal{C}\mathcal{I}(X,T) \,\subseteq \,\mathrm{Trans}(X,T).$$
\end{proposition}

\begin{remark} Let $Y$ be a locally compact Hausdorff space which is not compact and has no isolated points. Then, for any continuous map $T\colon Y \to Y$, the interior of the set $\mathrm{Trans}(X,T)$ is empty
(cf. \cite[Theorem 1]{Ber}). Moreover, if $X$ is a compact Hausdorff space without isolated points, then for any continuous map $T\colon X \to X$ either $\mathrm{Trans}(X,T) = X$ or the interior of the set $\mathrm{Trans}(X,T)$ is empty (cf. \cite[Theorem 5]{Ber}). In particular, if $X$ is a compact convex Hausdorff space without isolated points, and it is not a singleton, then for any continuous map $T\colon X \to X$ the interior of the set $\mathrm{Trans}(X,T)$ is empty, since $T$ has a fixed point by Schauder-Tychonoff Theorem.
\end{remark}

\section{Examples}\label{se:examples}

We will now apply our results to a few examples.

\begin{example}\label{ex:1} \textbf{Expanding maps and $\beta$-shifts.}

Let $X$ be a compact connected boundaryless Riemannian manifold and $T\colon X \to X$ a $C^{1}$ expanding map. It is known that the periodic points of $T$ are dense in $X$, so $T$ is not uniquely ergodic; every
$T$-invariant probability measure can be approximated in the weak$^*$ topology by probability measures supported on periodic orbits, thus $\mathcal{P}_{T}^{erg}(X)$ is dense in $\mathcal{P}_{T}(X)$;
and, by Ruelle's theorem, there is a $T$-invariant probability measure with full support, hence $\mathcal{M}(X,T) = X$ (cf. \cite{OV}). Therefore, we may apply Theorem~\ref{theorem-B} and deduce that both sets
$\big\{x\in X\colon \, V_T(x) = \mathcal{P}_{T}(X)\big\}$ and $\bigcap_{\varphi \,\in \,\mathcal{H}(X,T)} \mathcal{I}(T,\varphi)$ are Baire generic in $X$. This does not lead us beyond what is already known, though, since such a map $T$ is topologically mixing and satisfies the periodic specification property, so we may also use \cite[Corollary 1]{Tian} to conclude that the set of completely irregular points of $T$ is Baire generic in $X$.

On the contrary, not all $\beta$-shifts (or S-gap shifts, for that matter) satisfy the specification property (cf. \cite{Ren, Jung}). Yet, every $\beta$-shift $T_\beta\colon X_\beta \to X_\beta$ has the almost specification property, so the irregular set for any map $\varphi \in C^0(X, \mathbb{R})$ is either empty or has full topological entropy and Hausdorff dimension (cf. \cite{Thomp}); and the set of completely irregular points is Baire generic in $X_\beta$ (cf. \cite{Tian}). On the other hand, such a transformation $T_\beta$ is transitive, not uniquely ergodic, $\mathcal{M}(X_\beta,T) = X_\beta$ and $\mathcal{P}_{T}^{erg}(X_\beta)$ is dense in $\mathcal{P}_{T}(X_\beta)$ (cf. \cite{KK}). Therefore, we likewise deduce from Theorem~\ref{theorem-B} that both sets $\big\{x\in X\colon \, V_T(x) = \mathcal{P}_{T}(X)\big\}$ and $\bigcap_{\varphi \,\in \,\mathcal{H}(X,T)} \mathcal{I}(T,\varphi)$ are Baire generic in $X$.

\end{example}

\begin{example}\label{ex:1A} \textbf{Approximate product property.}

In \cite[page 2]{KLO16}, we may find an accurate scheme connecting several specification-like properties, such as almost specification, specification, periodic specification, weak-specification, periodic weak-specification,  relative almost specification, relative specification, relative weak-specification, almost product property and approximate product property, just to name a few. All the aforementioned properties imply the approximate product property (cf. \cite{DAT90,KPR14,PfSu07,YA09} for more details regarding this notion).

It is known that, for any continuous map $T\colon X \to X$ with the approximate product property, the set $\mathcal{P}_{T}^{erg}(X)$ is dense in $P_{T}(X)$. Actually, $\mathcal{P}_{T}^{erg}(X)$ is entropy-dense in $\mathcal{P}_{T}(X)$ -- see \cite[Proposition 2.1 (6) and Theorem E (3)]{HLT21v1} or \cite[Proposition 2.3 and Theorem 2.1]{PfSu}. Moreover, the approximate product property also implies that the set $\AP(X,T)$ of minimal points is dense in $\mathcal{M}(X,T)$ (cf. \cite[Lemma 5.8]{HLT21v2}). So, if $X$ is a compact metric space and $T\colon X \to X$ is a continuous map satisfying the approximate product property and such that $\mathcal{M}(X,T) = X$, then:

\begin{itemize}
\item We may apply Theorem~\ref{theorem-B} to conclude that $T$ is transitive and the set $\mathcal{C}\mathcal{I}(X,T)$ is Baire generic in $X$.
\smallskip

\item Item (d) of Corollary~\ref{cor:ergodic} informs that $\big\{x\in X\colon \, V_T(x) = \mathcal{P}_{T}(X)\big\}$ is Baire generic in $X$.
\smallskip

\item If, in addition, $T$ is non-uniquely ergodic and the compact metric space $X$ has no isolated points, then we can apply Proposition~\ref{theo_tra_sha_mod} to show that the orbit of every completely irregular point is dense in $X$.
\end{itemize}

For more information on systems with the approximate product property for which $M(X,T)=X$, we refer the reader to \cite[Theorem A, Remark 1.2 (b) and Theorem D(3)]{HLT21v2}.

After \cite[Theorem 1.4]{DTY15}, we are aware that, if $T\colon X \to X$ satisfies the asymptotic average shadowing property, then $\big\{x\in X\colon\, V_T(x) = \mathcal{P}_{T}(X)\big\}$ is nonempty. Therefore, Proposition~\ref{theo_tra_sha_mod} also applies to non-uniquely ergodic continuous maps, acting on compact metric spaces without isolated points, satisfying the asymptotic average shadowing property and such that the set $\AP(X,T)$ of minimal points is dense in $X$. In this context, we recall that Kulczycki and Oprocha showed that, if $T$ satisfies the asymptotic average shadowing property and the set of minimal points is dense in $X$, then $T$ is totally transitive (cf \cite[Theorem 3.4]{KO10}).

\end{example}

\begin{example}\label{ex:2} \textbf{Transitive non-uniquely ergodic subshifts.}

Let $\mathcal{A}$ be a finite alphabet with $\ell$ symbols, $X$ be the space $\mathcal{A}^{\mathbb{Z}}$ of all bilateral sequences of elements of $\mathcal{A}$ and $\sigma$ the shift operator on $X$. Since every subshift of finite type has the shadowing property (cf. \cite{W78}), then any transitive subshift of finite type which is not uniquely ergodic satisfies the hypothesis of Theorem~\ref{theorem-C}.

It is easy to exhibit an example of a continuous map $T\colon X \to X$ on a compact metric space $X$ which is not uniquely ergodic and whose set of irregular points for any continuous function $\varphi \colon X \to X$ is empty. For instance, the map $T(x) = x^2$ in the unit interval $[0,1]$ has exactly two $T$-invariant ergodic probabilities and every $x \in [0,1]$ is in the basin of attraction of one of them; in particular, any $T$-invariant non-ergodic probability measure has empty basin. In \cite[Proposition 9.9]{KK}, one may find the construction of a topologically mixing shift with the same properties.

\end{example}

\begin{example}\label{ex:3} \textbf{Conservative mixing homeomorphisms with the shadowing property.}

Let $(X,d)$ be a compact connected manifold, with or without boundary, of dimension $\mathrm{dim} \,X \geq 2$, endowed with a distance $d$. A \emph{Lebesgue-like measure} $\mu$ on $X$ is a Borel probability measure which is non-atomic, has full support and gives measure zero to the boundary $\partial X$ of $X$:
\begin{itemize}
\item[(i)]  $\,\,\forall\, p \in X \quad \mu(\left\{ p \right\} ) = 0$;
\item[(ii)] $\,\,\mu(U) > 0$ for every nonempty open set  $U \subset X$;
\item[(iii)] $\,\, \mu\,(\partial X) = 0$.
\end{itemize}

Suppose that $X$ possesses a Lebesgue-like measure, say $\mu_0$. Denote by $\Hom(X,\mu_0)$ the set of homeomorphisms of $X$ which preserve the measure $\mu_0$. This space is metrizable by the uniform distance for homeomorphisms, given by $D(f,g) = D_1(f,g) + D_1(f^{-1},g^{-1})$, where $D_1(f,g) = \max_{x \,\in\, X} d(f(x), g(x))$.

According to \cite[Theorem 1.3 and Corollary 1.4]{GL18}, a generic element in $\Hom(X,\mu_0)$ is topologically mixing and satisfies the shadowing property. Thus, the conclusion of Theorem~\ref{theorem-C} holds generically in $\Hom(X,\mu_0)$.

It is worthwhile comparing this result with \cite[Theorem 3.6]{AA13}, where it was proved that a generic homeomorphism of $X$ has convergent Birkhoff averages at Lebesgue almost everywhere for every continuous potential $\varphi \in C^0(X, \mathbb{R})$.

\end{example}

\begin{example}\label{ex:7} \textbf{Positively expansive maps on connected spaces.}

Let $T\colon X \to X$ be a continuous self-map of a compact metric space $(X,d)$. It is known that:
\begin{itemize}
\item[(a)] If $T$ is positively expansive, then $T$ is open if and only if $T$ has the shadowing property (cf. \cite[Theorem 1]{Sakai1}).
\smallskip
\item[(b)] If $X$ is connected and $T$ is positively expansive and open, then $T$ is transitive (more precisely, $T$ is topologically mixing -- see \cite{Ruelle, Sakai2}).
\end{itemize}

Assume now that $X$ is connected and $T\colon X \to X$ is a transitive continuous map with the shadowing property. Then $T$ is neither minimal nor uniquely ergodic. Indeed, either such a $T$ is minimal, equicontinuous and its topological entropy $h_{\mathrm{top}}(T) = 0$, or $T$ is not minimal and $h_{\mathrm{top}}(T) >0$ (cf. \cite[Theorem 6]{Moo11}). Moreover, when $T$ is equicontinuous, the space $X$ is totally disconnected (cf. \cite[Theorem 4]{Moo11}). Therefore, such a $T$ cannot be minimal.
Furthermore, by \cite[Theorem 1.5]{DOT18}, $\mathcal{H}(X,T) \neq \emptyset$; hence $T$ is not uniquely ergodic. 

Consequently, if $T\colon X \to X$ is a continuous, open, positively expansive map of a compact connected metric space $(X,d)$, then $T$ is transitive, has the shadowing property, is not uniquely ergodic and $X$ has no isolated points; thus, all the conclusions of Theorem~\ref{theorem-C} are valid for $T$.

\end{example}

\section{Expansive homeomorphisms}

Motivated by the statement of Theorem~\ref{theorem-C}, in the first half of this section we address the connection between the existence of irregular points, the cardinality of the non-wandering set and the topological entropy of expansive homeomorphisms with the shadowing property; whereas in the second half we present a counterpart of Proposition~\ref{sha_trans_equi} for this family of maps.

\subsection{Irregular points and entropy}\label{sse:ip-ent} The next result indicates that, within this setting, the existence of an irregular point has significant consequences.

\begin{proposition}\label{non_wand_card_entropy}
Let $(X,d)$ be a compact metric space without isolated points and $T\colon X \to X$ be an expansive homeomorphism with the shadowing property. Then the following assertions are equivalent:
\begin{itemize}
\item[$(i)$] $\,\,\mathcal{H}(X,T) \neq \emptyset$.
\smallskip
\item[$(ii)$] $\,\,h_{\mathrm{top}}(T) >0$.
\smallskip
\item[$(iii)$] $\,\,\#\,\Omega(T)$ is infinite.
\end{itemize}
\end{proposition}

For instance, an Axiom A diffeomorphism has positive topological entropy if and only if $\#\,\Omega(T)$ is infinite.

\begin{proof} $(i) \Leftrightarrow (ii)$ by \cite[Theorem 1.5]{DOT18}.

\noindent $(ii) \Rightarrow (iii)$ Suppose that $h_{\mathrm{top}}(T) >0$.
Then $\#\,\Omega(T)$ is infinite, otherwise $\Omega(T)$ is a finite union of periodic orbits, and therefore a power $T^p_{|\Omega(T)}$ is the identity map, which has zero entropy; hence
$$h_{\mathrm{top}}(T) \,=\, h_{\mathrm{top}}(T_{|\Omega(T)}) \, = \, \frac{1}{p}\,h_{\mathrm{top}}(T^p_{|\Omega(T)}) \,= \,0$$
contradicting the assumption.

\noindent $(iii) \Rightarrow (ii)$ Suppose that $\#\,\Omega(T)$ is infinite. As $X$ is a compact metric space and $T\colon X \to X $ is an expansive homeomorphism with the shadowing property, then $\Omega(T)$ can be written as a finite union of disjoint closed invariant sets
$(F_i)_{1 \,\leq \,i \,\leq \,K}$ on each of which $T$ is topologically transitive (see \cite[Theorem 11.13]{A93}), or \cite[Theorem 3.4.4]{AH94}, or \cite[Theorem 3]{DLRW13}).
Moreover, from \cite[Theorem 3.4.2]{AH94} one has $\overline{\mathrm{Per}(X,T)} = \Omega(T)$, so
$$\overline{\mathrm{Per}(X,T)} \,= \, F_{1} \cup \cdots \cup F_{k}.$$
Thus $\mathrm{Per}(X,T) \cap F_{i}$ is dense in $F_{i}$ for every $i \in \{1,\cdots,k\}$, since these subsets are closed, disjoint and $T$-invariant.

As $\#\,\Omega(T)$ is infinite, there exists $F_i$ whose cardinal in infinite. Relabeling the sets if necessary, we may assume that $\#F_{1}$ is infinite. Taking into account that $T|_{F_{1}} \colon F_{1}\to F_{1}$ is a transitive map with a dense set of periodic points and $\#F_{1}$ is infinite, we conclude that $(F_{1},T|_{F_{1}})$ has sensivity to initial conditions (cf. \cite{BBGDS92}). So, there exists a sensitive point in $\Omega(T)$. By \cite[Theorem 3.7]{LO13}, one has
$$h_{\mathrm{top}}(T_{|\Omega(T)}) \,=\, \max_{i\, \in \,\{1,\cdots,k\}} \,h_{\mathrm{top}}(T|_{F_{i}}) \, \geq \,  h_{\mathrm{top}}(T|_{F_{1}}) > 0.$$
\end{proof}

\begin{corollary}
Let $(X,d)$ be a compact metric space without isolated points and $T\colon X \to X$ be an expansive homeomorphism with the shadowing property. If $h_{\mathrm{top}}(T) >0$ then there exists a basic set $\Lambda \subseteq  \Omega(T)$ with infinite cardinal such that $\mathcal{C}\mathcal{I}(\Lambda, \,T_{|\Lambda}) \,= \bigcap_{\varphi \,\in \,\mathcal{H}(\Lambda, \,T_{|\Lambda})} \,{\mathcal I}( T_{|\Lambda},\,\varphi)$ is Baire generic in $\Lambda$.
\end{corollary}

\begin{proof}
Write, as in the previous argument, $\Omega(T) = F_{1} \cup \cdots \cup F_{k}$, where each $F_{i}$ is a closed $T$-invariant subset of $X$ for $i \in \{1,\cdots, k\}$, on each of which $T$ is topologically transitive and  such that $F_{i} \cap F_{j} = \emptyset$ for every $i\neq j$. As $h_{\mathrm{top}}(T, \Omega(T)) = \max_{i\, \in \,\{1,\cdots,k\}} h_{\mathrm{top}}(T|_{F_{i}}) > 0$, there exists $p \in \{1,\cdots,k\}$ such that $h_{\mathrm{top}}(T|_{F_{p}}) = h_{\mathrm{top}}(T) > 0$.

From \cite[Theorem 3.4.2]{AH94}, $T|_{\Omega(T)}\colon \Omega(T) \to \Omega(T)$ has the shadowing property. Therefore,
$T|_{F_{i}}\colon F_{i}\to F_{i}$ is an expansive homeomorphism with the shadowing property, for every $i \in \{1,\cdots,k \}$. As $h_{\mathrm{top}}(T|_{F_{p}})  > 0$, by Proposition~\ref{non_wand_card_entropy} one has $\mathcal{H}(F_{p},T|_{F_{p}}) \neq \emptyset$ and $\#\,\Omega(T|_{F_{p}})$ is infinite. Hence $\#\,F_{p}$ is infinite because $\Omega(T|_{F_{p}}) \subseteq F_{p}$. In particular, $T|_{F_{p}}$ is a transitive homeomorphism with the shadowing property such that $\mathcal{H}(F_{p},T|_{F_{p}}) \neq \emptyset$. Thus $T|_{F_{p}}$ is not uniquely ergodic. By Theorem~\ref{theorem-C}, the set $\mathcal{C}\mathcal{I}(F_{p},\, T|_{F_{p}}) \,= \bigcap_{\varphi \,\in \,\mathcal{H}(F_{p},\,T|_{F_{p}})} \,{\mathcal I}(T|_{F_{p}},\,\varphi)$ is Baire generic in $F_{p}$. So $\Lambda = F_p$ has the claimed properties.

\end{proof}

\subsection{Transitivity}

We now summon Corollary~\ref{cor:ergodic} and Proposition~\ref{sha_trans_equi} to check whether expansive homeomorphisms with the shadowing property may have points whose set $V_T$ is maximum.

\begin{proposition}\label{exp_sha_trans}
Let $T\colon X \to X$ be an expansive homeomorphism of a compact metric space $(X,d)$ without isolated points which satisfies the shadowing property. Then the following assertions are equivalent:
\begin{itemize}
\item[$(i)$] $T$ is transitive.
\smallskip
\item[$(ii)$] $\big\{x\in X\colon \, V_T(x) = \mathcal{P}_{T}(X)\big\}$ is Baire generic in $X$.
\smallskip
\item[$(iii)$] $\big\{x\in X\colon\, V_T(x) = \mathcal{P}_{T}(X)\big\}$ is nonempty.
\end{itemize}
\end{proposition}

\begin{proof}
$(i)\Rightarrow(ii)$ Suppose that $T$ is transitive. Then, as $X$ has no isolated points, $\Omega(T) = X$. By Lemma~\ref{lemma.center.measure.nonwandering}, we already know that $\mathcal{M}(X,T) \subseteq \Omega(T)$. Moreover, the set $\AP(X,T)$ of minimal points of $T$ is dense in $\Omega(T)$ because $T$ satisfies the shadowing property (cf. \cite[Corollary 1]{Moo11}). Since $\AP(X,T) \subseteq \mathcal{M}(X,T)$ and $\mathcal{M}(X,T)$ is closed, one concludes that $\mathcal{M}(X,T) = \Omega(T) = X$. Therefore, by Proposition~\ref{sha_trans_equi}, the set $\mathcal{P}_{T}^{erg}(X)$ is dense in $\mathcal{P}_{T}(X)$. Bringing together the properties $\mathcal{M}(X,T) = X$ and $\overline{\mathcal{P}_{T}^{erg}(X)} = \mathcal{P}_{T}(X)$, Corollary~\ref{cor:ergodic} yields that $\big\{x\in X\colon \, V_T(x) = \mathcal{P}_{T}(X)\big\}$ is Baire generic in $X$.

\smallskip

\begin{remark}\label{rem:Big VT} We note that, in the previous reasoning, we did not use the assumption that $T$ is an expansive homeomorphism. Actually, we have proved that, \emph{if $(X,T)$ is a compact metric space without isolated points and $T\colon X \to X$ is a transitive continuous map which satisfies the shadowing property, then $\big\{x\in X\colon \, V_T(x) = \mathcal{P}_{T}(X)\big\}$ is Baire generic in $X$}.
\end{remark}

\smallskip

\noindent $(ii)\Rightarrow(iii)$ This is clear.

\smallskip

\noindent $(iii)\Rightarrow(i)$ Suppose that $\big\{x\in X\colon\, V_T(x) = \mathcal{P}_{T}(X)\big\}$ is nonempty. Take $x_{0} \in X$ such that $V_T(x_{0}) = \mathcal{P}_{T}(X)$. Write, as in the argument used in Subsection~\ref{sse:ip-ent}, $\Omega(T) = F_{1} \cup \cdots \cup F_{k}$, where each $F_{i}$, for $i \in \{1,\cdots, k\}$, is a closed $T$-invariant subset of $X$, on each of which $T$ is topologically transitive and  such that $F_{i} \cap F_{j} = \emptyset$ for every $i\neq j$.

From \cite[Theorem 3.2.2]{AH94}, we know that there exists $F_{p}$ such that $\omega(x_{0}) \subseteq F_{p}$ for some $p \in \{1,\cdots,k\}$. Moreover, by item $(a)$ of Corollary \ref{cor:ergodic}, we have
$$\overline{\bigcup_{\mu\, \in \,\mathcal{P}_{T}(X)} \,\supp \mu}\,\subseteq\, \omega (x_{0}).$$

\begin{claim} $F_{j} = \emptyset$ for every $j \neq p$.
\end{claim}

If, otherwise, there exists $y \in F_{j}$ for some $j \neq p$, then, as $V_{T}(y) \neq \emptyset$, there exists an invariant measure $\eta$ such that $\supp \eta \subseteq \omega(y) \subseteq F_{j}$. However,
$$\supp \eta \,\subseteq\, \overline{\bigcup_{\mu\, \in \,\mathcal{P}_{T}(X)} \,\supp \mu}\,\subseteq\, \omega (x_{0}) \,\subseteq \,F_{p}$$
contradicting the fact that $F_{r} \cap F_{s} = \emptyset$ for every $r\neq s$.

Consequently, $\Omega(T) = F_{p}$ and $T_{F_{p}} = T_{\Omega(T)}$ is transitive. Since $T$ also satisfies the shadowing property, one has $C_R(X,T) = \Omega(T)$ (see \cite[Theorem 3.1.2]{AH94}). Hence $T|_{C_R(X,T)}$
is transitive. Therefore, $T$ is transitive in $X$ (cf. \cite[Theorem 1.1]{MM00}).

\end{proof}

\section{$\Phi$-chaotic dynamics}\label{se:chaotic}

Given a metric space $(Y,d)$, denote by $C^b(Y,\mathbb R)$ the set of continuous bounded real valued maps on $Y$ endowed with the supremum norm $\|\cdot\|_\infty$. Consider a sequence $\Phi = (\varphi_n)_{n \, \in\, \mathbb{N}} \in C^b(Y,\mathbb R)^{\mathbb N}$ and $y \in Y$. In what follows, $W_{\Phi}(y)$ stands for the set of accumulation points of the sequence $(\varphi_{n}(y))_{n\, \in \,\mathbb{N}}$.

\begin{definition}\cite[Definition 1]{CCSV21}\label{defTphisensitive}
\emph{Let $(Y,d)$ be a metric space and $\Phi \in C^b(Y,\mathbb R)^{\mathbb N}$. We say that $Y$ is $\Phi$-\emph{sensitive} if there exist dense subsets $A,\,B \subset Y$, where $B$ can be equal to $A$, and $\varepsilon > 0$ such that for any $(a,b) \in A \times B$ one has
$$\sup_{r\,\in \,W_{\Phi}(a), \,\, s \,\in \,W_{\Phi}(b)} \,\,|r-s| > \varepsilon.$$
In the particular case of the sequence $\Phi$ of Ces\`aro averages associated with a map $\varphi \in C^b(Y,\mathbb R)$ and a continuous transformation $T\colon Y \to Y$, namely
$$(\varphi_{n})_{n\, \in \,\mathbb{N}} \,=\, \Big(\frac{1}{n} \,\sum_{j=0}^{n-1} \,\varphi\circ T^{j}\Big)_{n \, \in \, \mathbb{N}}$$
we say that $Y$ is $(T,\varphi)$-\emph{sensitive} if $Y$ is $\Phi$-\emph{sensitive}, and write $W_{\varphi}$ instead of $W_{\Phi}$.}
\end{definition}

In \cite[Theorem 1.2]{CCSV21}, it was proved that if $Y$ is $\Phi$-\emph{sensitive} then the set
$${\mathcal I}(\Phi) \,=\, \Big\{y \in Y \colon \,\lim_{n\,\to \,+\infty} \,\varphi_n(y) \,\, \text{ does not exist} \,\Big\}$$
of the so called $\Phi$-irregular points is Baire generic in $Y$.

The previous definition inspired the following one. Denote by $\mathfrak{P}(Y)$ the family of all subsets of $Y$ and by $\diam(A)$ the diameter of $A \subseteq X$.

\begin{definition}\label{defTphichaotic}
\emph{Given $\Phi \in C^b(Y,\mathbb R)^{\mathbb N}$, we say that $Y$ is $\Phi$-\emph{chaotic} if there exist $\mathfrak{F} \subseteq  \mathfrak{P}(Y)$ and $\varepsilon > 0$ such that for every nonempty open set $U$ in $Y$ there exist $a, b \in  \bigcup_{F \,\in\, \mathfrak{F}} F \cap U$, where $a$ and $b$ may be equal, such that
$$\diam(W_{\Phi}(a) \cup W_{\Phi}(b)) \,>\, \varepsilon.$$
In the particular case of a sequence $\Phi$ of Ces\`aro averages
$$(\varphi_{n})_{n\, \in \,\mathbb{N}} \,=\, \Big(\frac{1}{n} \,\sum_{j=0}^{n-1} \,\varphi\circ T^{j}\Big)_{n \, \in \, \mathbb{N}}$$
determined by a map $\varphi \in C^b(Y,\mathbb R)$ and a continuous transformation $T\colon Y \to Y$, we say that $Y$ is $(T,\varphi)$-\emph{chaotic} if the space $Y$ is $\Phi$-\emph{chaotic}, and write $W_{\varphi}$ instead of $W_{\Phi}$.}
\end{definition}

We note that $Y$ is $(T,\varphi)$-chaotic if and only if there exist a dense set $\mathcal{E}$ in $Y$ and $\varepsilon > 0$ such that for every nonempty open set $U$ in $Y$ there are $a, b \in  \mathcal{E} \cap U$, where $a$ and $b$ may be equal, such that $\diam(W_{\Phi}(a) \cup W_{\Phi}(b)) > \varepsilon.$ Indeed, the latter statement clearly implies Definition~\ref{defTphichaotic} if we use $\mathfrak{F} = \{\mathcal{E}\}$. Conversely, if $Y$ is $\Phi$-\emph{chaotic}, then we may take $\mathcal{E} = \bigcup_{F \,\in\, \mathfrak{F}} F$, which is dense in $Y$ since it intersects every nonempty open set $U \subset Y$, and there are $a, b \in  \mathcal{E} \cap U$, where $a$ and $b$ may be equal, such that $\diam(W_{\Phi}(a) \cup W_{\Phi}(b)) > \varepsilon$, as claimed.

\begin{theorem}\label{CV_mod_chaotic}
Let $(Y, d)$ be a Baire metric space and $\Phi = (\varphi_n)_{n \, \in\, \mathbb{N}} \in C^b(Y,\mathbb R)^{\mathbb N}$ be a sequence of continuous bounded maps such that $\limsup_{n\,\to\,+\infty} \,\|\varphi_n\|_\infty < +\infty$. If $Y$ is $\Phi$-chaotic then the irregular set of $\Phi$, defined by
$${\mathcal I}(\Phi) \,=\, \Big\{y \in Y \colon \,\lim_{n\,\to \,+\infty} \,\varphi_n(y) \,\, \text{ does not exist} \,\Big\}$$
is Baire generic in $Y$. In particular, if $T\colon Y \to Y$ is a continuous map, $\varphi$ belongs to $C^b(Y,\mathbb R)$ and $Y$ is $(T,\varphi)$-chaotic, then ${\mathcal I}(T,\varphi)$ is a Baire generic subset of $Y$.
\end{theorem}

\begin{proof}

The following argument is a direct adaptation of the one used to show \cite[Theorem 1.2]{CCSV21}. Suppose that there exist $\mathfrak{F} \subseteq  \mathfrak{P}(Y)$ and $\varepsilon > 0$ such that for any open set $U$ in $Y$ there exist $a, b \in  \bigcup_{F \,\in\, \mathfrak{F}} F \cap U$, where $a$ and $b$ may be equal, such that
$$\diam(W_{\Phi}(a) \cup W_{\Phi}(b)) > \varepsilon.$$
Fix $0 <\eta < \frac{\varepsilon}{3}$. Since the maps $\varphi_n$ are continuous, given an integer $N \in \mathbb{N}$ the set
$$\Lambda_{N} =  \Big\{y \in Y\colon \,|\varphi_{n}(y) - \varphi_{m}(y) | \leqslant \eta \quad \forall  \,m,n \geqslant N \Big\}$$
is closed in $Y$. Moreover:

\begin{lemma}\label{eq:interior.empty}
$\Lambda_{N}$ has empty interior for every $N \in \mathbb{N}$.
\end{lemma}

\begin{proof}
Assume that there exists $N \in \mathbb{N}$ such that $\Lambda_{N}$ has nonempty interior (which we abbreviate into $\intt (\Lambda_{N})\neq \emptyset$). Hence there exist $y_0$ and $\xi>0$ such that $B(y_0,\xi)  \subseteq   \intt (\Lambda_{N})$, where $B(y_0,\xi)=\{y \in Y\colon d(y,b) < \xi\}$. Since $\varphi_{N}$ is continuous in $y_0$, there exists $\delta_{N} >0$ such that
$$y \in B(y_0,\delta_{N}) \quad \Rightarrow \quad |\varphi_{N}(y_0) - \varphi_{N}(y)| <  \eta/3 \quad \quad \forall\, y \in Y.$$
Take $r = \min \{\xi, \delta_{N}\}$. Then,
$$B(y_0,r) \, \subseteq \,B(y_0,\xi) \, \subseteq \,  \intt (\Lambda_{N})$$
and
$$|\varphi_{N}(y) - \varphi_{N}(z)| <  \eta \quad \quad \forall\, y,z \in B(y_0,r).$$
By assumption, there are $\varepsilon>0$ and $a, b \in  \bigcup_{F \,\in \, \mathfrak{F}} \big(F\cap B(y_0,r)\big)$, where $a$ and $b$ may be equal, such that
$$\diam(W_{\Phi}(a) \cup W_{\Phi}(b)) > \varepsilon.$$
Moreover, as $a, b \in B(y_0,r)$, one has $|\varphi_{N}(a) - \varphi_{N}(b)| < \eta$.
\medskip

\noindent \textbf{Case 1:} There exists $(r_{a},r_{b}) \in  W_{\Phi}(a) \times W_{\Phi}(b)$
satisfying $|r_{a} - r_{b}| > \varepsilon$.
\medskip

According to the definition of $\Lambda_{N}$, one has
$$|\varphi_{n}(a) - \varphi_{m}(a) | \leqslant \eta \quad \quad \text{ and } \quad \quad |\varphi_{n}(b) - \varphi_{m}(b)| \leqslant \eta \quad \quad \forall\, m, n \geqslant N.$$
Fixing $m = N$, taking the limit as $n$ goes to $+\infty$ in the first inequality along a subsequence of $\big(\varphi_{n}(a)\big)_{n \, \in \, \mathbb{N}}$ converging to $r_{a}$ and taking the limit as $n$ tends to $+\infty$ in the second inequality along a subsequence of $\big(\varphi_{n}(b)\big)_{n \, \in \, \mathbb{N}}$ convergent to $r_{b}$, we conclude that
$$|r_{a} - \varphi_{N}(a)| \leqslant \eta \quad \quad \text{ and }\quad \quad |r_{b} - \varphi_{N}(b)| \leqslant \eta.$$
Therefore,
$$\varepsilon \,<\, |r_{a} - r_{b}| \,\leqslant\, |r_a - \varphi_{N}(a)| +  |\varphi_{N}(a) - \varphi_{N}(b)|+ |\varphi_{N}(b) - r_{b}| \,\leqslant\, 3 \eta$$
contradicting the choice of $\eta$. So, $\Lambda_{N} $ have empty interior.
\medskip

\noindent \textbf{Case 2:} There does not exist $(r_{a},r_{b}) \in  W_{\Phi}(a) \times W_{\Phi}(b)$ satisfying $|r_{a} - r_{b}| > \varepsilon$.
\medskip

Since $\diam(W_{\Phi}(a) \cup W_{\Phi}(b)) > \varepsilon$, either there exist $s_a, t_a \in  W_{\Phi}(a)$ such that $|s_a-t_a| > \varepsilon$, or there are $s_{b}, t_b \in  W_{\Phi}(b)$ such that $|s_b-t_b| > \varepsilon$. Assume that the former holds (the reasoning is similar if the latter holds instead).
According to the definition of $\Lambda_{N}$, one has
$$|\varphi_{n}(a) - \varphi_{m}(a) | \leqslant \eta.$$
Fixing $m = N$, taking the limit as $n$ goes to $+\infty$ in this inequality along a subsequence of $\big(\varphi_{n}(a)\big)_{n \, \in \,\mathbb{N}}$ converging to $s_a$ and taking the limit as $n$ tends to $+\infty$ in the same inequality along a subsequence of $\big(\varphi_{n}(a)\big)_{n \, \in \, \mathbb{N}}$ convergent to $t_a$, we conclude that
$$|s_a - \varphi_{N}(a)| \leqslant \eta \quad \quad \text{ and }\quad \quad |t_a - \varphi_{N}(a)| \leqslant \eta.
$$
Therefore,
$$\varepsilon \,<\, |s_a - t_a| \,\leqslant\, |s_a -\varphi_{N}(a)| + |\varphi_{N}(a) - t_a|  \,\leqslant\, 2 \eta$$
contradicting the choice of $\eta$. Again, this conclusion indicates that $\Lambda_{N} $ has empty interior.
\end{proof}

Let us resume the proof of Theorem~\ref{CV_mod_chaotic}. From the assumption $\limsup_{n\,\to\,+\infty} \|\varphi_n\|_\infty <+\infty$ one deduces that
$$Y \setminus {\mathcal I}(\Phi) \,\subset\, \bigcup_{N=1}^{\infty} \Lambda_{N}.$$
Thus, by Lemma~\ref{eq:interior.empty}, the set of $\Phi$-regular points (that is, those points for which $\Phi$ is a convergent sequence) is contained in a countable union of closed sets with empty interior. Thus ${\mathcal I}(\Phi)$ is Baire generic in $Y$. The second statement of Theorem~\ref{CV_mod_chaotic} is just a particular case of the first one.
\end{proof}

Given $y_0 \in Y$, let $\mathcal O_T(y_0)^- = \,\big\{y \in Y \colon \exists \, n \in \mathbb{N} \colon T^n(y) = y_0 \big\}$ be the pre-orbit of $y_0$ by $T$. We say that $x \in Y$ is $T$-associated with $y \in Y$, and denote this equivalence relation by $x \sim_T y$, if either $y \in  \mathcal{O}_T(x)$ or $x \in \mathcal{O}_T(y)$. We note that, if $[x] = \{y \in Y \colon x \sim y \}$, then
$$\mathcal{O}_T(x) \cup \mathcal O_T(x)^-  \,=\, [x].$$
Regarding these equivalence classes, Theorem~\ref{CV_mod_chaotic} has the following immediate consequence which improves the statements 
of \cite[Corollary 1.7]{CCSV21} and \cite[Remark 10.1]{CCSV21}.

\begin{corollary}\label{cor_irregular_dense}
Let $(Y, d)$ be a Baire metric space, $T\colon Y \to Y$ be a continuous map and $\varphi \in C^{b}(Y,\mathbb{R})$. Then:
\begin{itemize}
\item[$(a)$] If there exist a dense set $D$ in $Y$ and $\varepsilon>0$ such that $\diam(W_{\Phi}(x)) > \varepsilon$ for every $x$ in $D$, then ${\mathcal I}(T,\varphi)$ is Baire generic in $Y$.
\smallskip
\item[$(b)$] If there exist $x_{1}, \cdots, x_{n} \in I(T,\varphi)$ such that $\bigcup_{i=1}^{n} \,[x_{i}]$ is dense in $Y$, then ${\mathcal I}(T,\varphi)$ is Baire generic in $Y$.
\end{itemize}
\end{corollary}

\begin{proof}
Item $(a)$ is clear: the assumption on the existence and properties of $D$ is a particular case of Definition~\ref{defTphichaotic}, thus Theorem~\ref{CV_mod_chaotic} holds if we take $\mathfrak{F} = \{D\}$ and $\varepsilon$. Concerning item $(b)$, consider
$$\varepsilon \,=\, \min\big\{\diam (W_{\varphi}(x_{i}))\colon  1 \,\leq\, i \,\leq \,n\big\}$$
which is positive since each $x_{i}$ is in $I(T,\varphi)$ and the set $\{x_{1}, \cdots, x_{n}\}$ is finite. Then apply Theorem~\ref{CV_mod_chaotic} using the set $\mathfrak{F} = \big\{[x_{1}], \cdots, [x_{n}]\big\}$ and the previous $\varepsilon$.
\end{proof}

\subsection{$\Phi$-sensitive \emph{vs.} $\Phi$-chaotic}

It is immediate that, if $Y$ is $\Phi$-sensitive, then it is $\Phi$-chaotic. Indeed, given $\varepsilon >0$ and dense subsets $A$ and $B$ satisfying Definition~\ref{defTphisensitive}, we just have to take $\mathfrak{F} = \{A, B\}$ and $\varepsilon$ to verify the definition of $\Phi$-chaotic.

In case $(Y, d)$ is a Baire metric space, $T\colon Y \to Y$ is a continuous map with a dense orbit and $\varphi \in C^b(Y,\mathbb R)$, we already know from \cite[Corollary 1.7]{CCSV21} that $Y$ is $(T,\varphi)$-sensitive if and only if the set ${\mathcal I}(T,\varphi)$ is Baire generic in $Y$. Since, by Theorem~\ref{CV_mod_chaotic}, the set ${\mathcal I}(T,\varphi)$ is also Baire generic when $Y$ is $(T,\varphi)$-chaotic, we conclude that, when $T$ has a dense orbit, then the notions $(T,\varphi)$-sensitive and $(T,\varphi)$-chaotic are equivalent. We will now prove that the properties $\Phi$-sensitive and $\Phi$-chaotic are actually equivalent.

\begin{proposition}
Let $(Y, d)$ be a Baire metric space and $\Phi = (\varphi_n)_{n \, \in \,\mathbb{N}} \in C^b(Y,\mathbb R)^{\mathbb N}$ be a sequence of continuous bounded maps such that $\limsup_{n\,\to\,+\infty} \,\|\varphi_n\|_\infty < +\infty$. Then $Y$ is $\Phi$-chaotic if and only if $Y$ is $\Phi$-sensitive.
\end{proposition}

\begin{proof}
It is already clear that if $Y$ is $\Phi$-sensitive then $Y$ is $\Phi$-chaotic. Suppose now that $Y$ is $\Phi$-chaotic, that is, there exist a dense set $\mathcal{E}$ in $Y$ and $\varepsilon > 0$ such that for every nonempty open set $U$ in $Y$ there are $a, b \in  \mathcal{E} \cap U$, where $a$ and $b$ may be equal, such that $\diam(W_{\Phi}(a) \cup W_{\Phi}(b)) > \varepsilon.$

We start by observing that the argument to prove Theorem~\ref{CV_mod_chaotic} also shows that
$$\Big\{y \in Y\colon \, \limsup_{n \, \to \, + \infty}\, \varphi_{n}(y) - \liminf_{n \, \to \, + \infty}\, \varphi_{n}(y) \,< \,\eta \Big\}  \,\subseteq\, \bigcup_{N=1}^{+\infty} \,\Lambda_{N}.$$
Therefore, by Lemma~\ref{eq:interior.empty}, the set
$$A \, = \, \big\{y \in Y \colon \, \limsup_{n \, \to \, + \infty}\, \varphi_{n}(y) - \liminf_{n \, \to \, + \infty}\, \varphi_{n}(y) \,\geqslant\, \eta \big\}$$
is Baire generic in $Y$. We are left to show that this property implies that $Y$ is $\Phi$-sensitive.

Consider $a, \,b$ in the dense set $A$ and $\varepsilon = \eta/4$. Given $\gamma_b \in W_\Phi(b)$, one has:

\begin{itemize}
\item[(i)] Either $\gamma_b \,\leqslant\, \liminf_{n \, \to \, + \infty}\, \varphi_{n}(a)$, in which case
$$\big|\limsup_{n \, \to \, + \infty}\, \varphi_{n}(a) - \gamma_b\big| \,\geqslant\, \eta$$
since $\diam (W_\Phi(a)) \geqslant\, \eta$;
\smallskip

\item[(ii)] or $\gamma_b \,\geqslant\, \limsup_{n \, \to \, + \infty}\, \varphi_{n}(a)$, and similarly
$$\big|\liminf_{n \, \to \, + \infty}\, \varphi_{n}(a) - \gamma_b\big| \,\geqslant\, \eta;$$
\smallskip

\item[(iii)] or else
$$\liminf_{n \, \to \, + \infty}\, \varphi_{n}(a) \, < \, \gamma_b \, < \, \limsup_{n \, \to \, + \infty}\, \varphi_{n}(a)$$
and so
$$\big|\liminf_{n \, \to \, + \infty}\, \varphi_{n}(a) - \gamma_b\big| \,\geqslant\, \eta/3 \quad \quad \text{or} \quad \quad \big|\limsup_{n \, \to \, + \infty}\, \varphi_{n}(a) - \gamma_b\big| \,\geqslant\, \eta/3.$$
\end{itemize}

\noindent In all cases, we conclude that there are $r_a \in W_\Phi(a)$ and $r_b \in W_\Phi(b)$ such that $|r_a - r_b| > \varepsilon$. This completes the proof that $Y$ is $\Phi$-sensitive.
\end{proof}

\section{Proof of Theorem~\ref{theorem-D}}\label{se:pD}

We start by establishing a general result, from which we easily deduce a dynamical version.

\begin{theorem}\label{Theorem-Psi_f}
Let $(Y,d)$ be a Baire metric space, $(W,D)$ be a compact metric space, $\Phi = (\varphi_{n})_{n\, \in \mathbb{N}}$ be a sequence of continuous maps $\varphi_n\colon Y \to \mathbb{R}$ and $\mathcal{G}_{\Phi}\colon Y \to \mathcal{K}(W)$ be the set valued function induced by
$\Phi$, defined by
$$\mathcal{G}_{\Phi}(y)\, = \,\bigcap\limits_{n=1}^{+\infty}  \overline{\{\varphi_{k}(y)\colon \,k \geq n\}}.$$
Then there exists a Baire generic subset $S$ of $Y$ such that $\mathcal{G}_{\Phi}$ is continuous in $S$.
\end{theorem}

Let $(Y, d)$ be a Baire metric space, $T\colon Y \to Y$ be a continuous function and consider $\varphi \in C^b(Y,\mathbb R)$. Recall from Definition~\ref{def:Wphi} that, for every $y \in Y$,
$$W_{\varphi}(y) \, \,= \,\, \bigcap\limits_{N=1}^{+\infty} \, \overline{\Big\lbrace  \frac{1}{n}\,\sum\limits_{j=0}^{n-1} \,\varphi(T^{j}(y))\colon \, n \geq N \Big\rbrace}.$$
Theorem~\ref{Theorem-Psi_f} may be rephrased in this setting as follows.

\begin{corollary}\label{Cor_Psi}
Let $(Y, d)$ be a Baire metric space, $T\colon Y \to Y$ be a continuous map and $\varphi \in C^b(Y,\mathbb R)$. There exists a Baire generic subset $S$ of $Y$ such that the map $W_{\varphi}$ is continuous in $S$.
\end{corollary}

The relevance of Corollary~\ref{Cor_Psi} is due to the fact that, whenever the map $T$ is transitive, we may detect the existence of a Baire generic subset of irregular points with respect to a fixed potential $\varphi \in C^b(Y,\mathbb R)$ through continuity properties of the map $W_\varphi$ -- and this is precisely the statement of Theorem~\ref{theorem-D}.

\begin{corollary}\label{Cor-T-Psi}
Let $(Y, d)$ be a Baire metric space, $T\colon Y \to Y$ be a transitive continuous map and $\varphi \in C^b(Y,\mathbb R)$. The following assertions are equivalent:
\begin{enumerate}
\item[$(a)$] ${\mathcal I}(T,\varphi)$ is Baire generic in $Y$.
\smallskip
\item[$(b)$] There exists a dense subset $\mathcal{Z}$ of $Y$ such that $\mathcal{Z} \cap {\mathcal I}(T,\varphi) \neq \emptyset$ and $W_{\varphi}$ is lower semi-continuous in $\mathcal{Z}$.
\end{enumerate}
\end{corollary}

\subsection{Set valued functions}\label{sse:svf}

In this section we gather information from \cite{WIN15} on set valued functions. The content of this reference was stated for compact metric spaces, and one needs to adapt part of it to Baire metric spaces. This reformulation, though not straightforward, is nevertheless easy to establish, and we leave the details for the interested reader.

Let $(Y,d)$ and $(W,D)$ be metric spaces. Given $E \subset W$ and $\varepsilon >0$, denote by $E^\varepsilon$ the union $\bigcup_{e \,\in\, E} B(e,\varepsilon)$, where $B(e,\varepsilon) = \{w \in W\colon D(e,w) < \varepsilon\}$.
Given non-empty subsets $A$ and $B$ of $Y$, define $d_{0}(A,B) = \sup_{b \,\in\, B} D(A, b)$, where $D(A,b) = \inf\,\{d(a,b) \colon \, a \in A\}$. Recall that $\mathcal{K}(W)$ stands for the collection of all nonempty compact subsets of $W$ endowed with the Hausdorff metric $d_{H}$.
Note that, if $(W,D)$ is compact, then $(\mathcal{K}(W), \,d_{H})$ is a compact metric space as well.

A function $\Psi\colon Y \to \mathcal{K}(W)$ is called a \emph{set valued function} on $Y$. Recall that such a map $\Psi$ is upper semi-continuous at $y_0 \in Y$ if given $\varepsilon >0$ there is $\delta>0$ such that $\Psi(y) \subseteq \Psi(y_0)^{\varepsilon}$ for every $y \in B(y_0,\delta)$. It is \emph{lower semi-continuous} at $y_0 \in Y$ if for every $\varepsilon >0$  there exists $\delta>0$ such that for every $y \in B(y_0,\delta)$ one has $\Psi(y_0) \subseteq \Psi(y)^{\varepsilon}$.
It is known that, if $\Psi$ is upper semi-continuous (or lower semi-continuous) then there exists a Baire generic subset $S$ of $Y$ such that $\Psi$ is continuous in $S$ (cf. \cite{FORT51}).

\begin{definition}
\emph{Suppose that $(W,D)$ is compact. Given a set valued function $\Psi\colon Y \to \mathcal{K}(W)$, the map
\begin{eqnarray*}
\widehat{\Psi}\colon Y &\to&  \mathcal{K}(W)\\
y &\mapsto& \bigcap_{n=1}^{+\infty} \,\overline{\bigcup_{a \,\in \,B(y,\frac{1}{n})} \Psi(a)}
\end{eqnarray*}
is called the \emph{upper regularization} of $\Psi$. Note that $\Psi(y) \subseteq \widehat{\Psi}(y)$ for every $y$ in $Y$.}
\end{definition}

\begin{lemma}[\cite{WIN15}]\label{Lemma-Upper}
Let $\Psi\colon Y \to \mathcal{K}(W)$ be an arbitrary set valued function and $\widehat{\Psi}$ be its upper regularization. Then:
\begin{itemize}
\item [$(a)$] $\widehat{\Psi}\colon Y \to  \mathcal{K}(W)$ is upper semi-continuous.
\item [$(b)$] If $\Psi$ is upper semi-continuous, then $\Psi = \widehat{\Psi}$.
\end{itemize}
\end{lemma}

\begin{definition}
\emph{For each $L \in \mathcal{K}(W)$, set:
\begin{eqnarray*}
J(L, \Psi) &=& \big\{y \in Y\colon \, \Psi(y) \cap L \neq \emptyset\big\}\\
R_{L} &=&
\big(Y \setminus \overline{J(L, \Psi)} \big) \cup J(L, \Psi)\\
R_{\Psi} &=& \bigcap\limits_{L \,\in\, \mathcal{K}(W)}\,  R_{L}
\end{eqnarray*}}
\end{definition}

\begin{lemma}[\cite{WIN15}]\label{Lemma-regular-upper}
Let $(Y,d)$ be a Baire metric space, $(W,D)$ be a compact metric space and $\Psi\colon Y \to \mathcal{K}(W)$ be a set valued function. If $y_0$ belongs to $R_{\Psi}$, then $\Psi$ is upper semi-continuous at $y_0$.
\end{lemma}

\begin{definition}
\emph{Let $\Phi = (\varphi_{n})_{n \,\in\, \mathbb{N}}$ be a sequence of continuous maps $\varphi_{n}\colon Y \to W$. The \emph{set valued function induced by $\Phi$}, which we denote by $\mathcal{G}_\Phi\colon Y \to \mathcal{K}(W)$, is given by
$$\mathcal{G}_\Phi(y) \,= \,\bigcap\limits_{n=1}^{+\infty} \, \overline{\{\varphi_{k}(y)\colon \, k \geq n\}}.$$}
\end{definition}

\begin{lemma}[\cite{WIN15}]\label{lemma-R_A_Baire}
Let $(Y,d)$ be a Baire metric space, $(W,D)$ be a compact metric space, $\Phi = (\varphi_{n})_{n \,\in\, \mathbb{N}}$ be a sequence of continuous maps $\varphi_{n}\colon Y \to W$, $\mathcal{G}_\Phi\colon Y \to \mathcal{K}(W)$ be the set valued function induced by $\Phi$ and $L \in \mathcal{K}(W)$. Then both $R_{L}$ and $R_{\mathcal{G}_\Phi}$ are Baire generic in $Y$.
\end{lemma}

\subsection{Proof of Theorem~\ref{Theorem-Psi_f}}

Let $(Y,d)$ be a Baire metric space, $(W,D)$ be a compact metric space, $\Phi = (\varphi_{n})_{n\, \in\, \mathbb{N}}$ be a sequence of continuous maps $\varphi_n\colon Y \to \mathbb{R}$ and $\mathcal{G}_{\Phi}\colon Y \to \mathcal{K}(W)$ be the set valued function induced by
$\Phi$. By Lemma~\ref{Lemma-regular-upper} and Lemma~\ref{lemma-R_A_Baire},
there exists a Baire generic subset $S_{1}$ of $Y$ such that $\mathcal{G}_{\Phi}$ is upper semi-continuous in $S_{1}$. Consider the upper regularization $\widehat{\mathcal{G}_{\Phi}}$ of $\mathcal{G}_{\Phi}$. By item $(a)$ of Lemma~\ref{Lemma-Upper}, the map
$\widehat{\mathcal{G}_{\Phi}}$ is upper semi-continuous in $Y$, and so
there exists a Baire generic subset $S_{2}$ of $Y$ such that $\widehat{\mathcal{G}_{\Phi}}$ is continuous in
$S_{2}$ (cf. \cite{FORT51}). Therefore, $S = S_{1} \cap S_{2}$ is Baire generic in $Y$, and one has both $\widehat{\mathcal{G}_{\Phi}}$ continuous in $S$ and $\mathcal{G}_{\Phi}$ upper semi-continuous in $S$. Consequently, from item $(b)$ of Lemma~\ref{Lemma-Upper}, we know that $\mathcal{G}_{\Phi} = \widehat{\mathcal{G}_{\Phi}}$ in $S$. Hence $\mathcal{G}_{\Phi}$ is continuous in $S$.

\subsection{Proof of Corollary~\ref{Cor-T-Psi}} Let $(Y, d)$ be a Baire metric space, $T\colon Y \to Y$ be a transitive continuous map and $\varphi \in C^b(Y,\mathbb R)$.
\medskip

\noindent $(a) \Rightarrow (b)$: By Corollary~\ref{Cor_Psi}, there exists a Baire generic subset $S$ of $Y$ such that $W_{\varphi}$ is continuous in $S$. Thus, if ${\mathcal I}(T,\varphi)$ is Baire generic in $Y$, then so is $\mathcal{Z} = S \cap {\mathcal I}(T,\varphi)$ 
and $W_{\varphi}$ is continuous in $\mathcal{Z}$.

\noindent $(b) \Rightarrow (a)$:
Suppose that there is a dense subset $\mathcal{Z}$ of $Y$ such that $\mathcal{Z} \cap {\mathcal I}(T,\varphi) \neq \emptyset$ and $W_{\varphi}$ is lower semi-continuous in $\mathcal{Z}$. Fix $y_0 \in \mathcal{Z}\cap {\mathcal I}(T,\varphi)$ and consider $\varepsilon_0 = \diam W_{\varphi}(y_0)/3.$ Note that $\varepsilon_0 > 0$ since $y_0 \in {\mathcal I}(T,\varphi)$. Therefore, there exists $\delta_0>0$ such that
$$\forall\, y \in B(y_0,\delta_0) \cap S \quad \quad W_{\varphi}(y_0) \subseteq W_{\varphi}(y)^{\varepsilon_0}.$$
Consequently,
$$\diam W_{\varphi}(y_0) \,\leq\, \diam W_{\varphi}(y) + 2\varepsilon_0 \,=\, \diam W_{\varphi}(y) + \frac{2}{3}\diam W_{\varphi}(y_0) \quad  \forall \,y \in B(y_0,\delta_0)\cap \mathcal{Z}$$
and so
\begin{equation}\label{eq:ineq}
\diam W_{\varphi}(y) \,\geq\, \frac{1}{3}\diam W_{\varphi}(y_0) \,= \, \varepsilon_0 \, > \, 0 \quad \forall \,y \in B(y_0,\delta_0)\cap \mathcal{Z}.
\end{equation}

\begin{claim} $Y$ is $(T,\varphi)$-chaotic.
\end{claim}

\begin{proof} We will to show that there exist a dense subset $\mathcal{E}$ in $Y$ and $\varepsilon>0$ such that for any open set $U$ in $Y$ there exists $a, b \in U \cap \mathcal{E}$, where $a$ and $b$ may coincide, satisfying $\diam \big(W_{\varphi}(a) \cup W_{\varphi}(b)\big) > \varepsilon$.

Consider $\mathcal{E} = \bigcup_{z \,\in \,\mathcal{Z}} \mathcal{O}_T(z)$, which is dense in $Y$, and $\varepsilon = \varepsilon_0/2$. Let $U$ be an open set of $Y$. Since $T$ is transitive, there exists $k>0$ such that $T^{-k}(U) \cap B(y_0,\delta_0) \neq \emptyset$. Moreover, as $T^{-k}(U)$ is open and $\mathcal{Z}$ is dense in $Y$, there exists $z_0 \in T^{-k}(U) \cap B(y_0,\delta_0) \cap \mathcal{Z}$.
Such a $z_0$ belongs to $B(y_0,\delta_0) \cap \mathcal{Z}$, so \eqref{eq:ineq} yields
$$\diam W_{\varphi}(z_0) \, \geq \, \varepsilon_0.$$
Moreover, $W_{\varphi}(z_0) = W_{\varphi}(T^{k}(z_0))$, hence
$\diam W_{\varphi}(T^{k}(z_0)) \geq \varepsilon_0$ as well. Therefore, as $T^{k}(z_0) \in U \cap \mathcal{E}$, we may take $a = b = T^{k}(z_0)$ and thereby confirm that $Y$ is $(T,\varphi)$-chaotic.
\end{proof}

Since $Y$ is $(T,\varphi)$-chaotic, we may apply Theorem~\ref{CV_mod_chaotic} and thus conclude that the set ${\mathcal I}(T,\varphi)$ is Baire generic in $Y$. The proof of $(b) \Rightarrow (a)$ is complete.

\subsection*{Acknowledgments}{LS was partially supported by FAPERJ - Aux\'ilio B\'asico \`a Pesquisa (APQ1), Project E-26/211.690/2021, FAPERJ - Jovem Cientista do Nosso Estado (JCNE) grant E-26/200.271/2023 and CAPES-Finance Code 001. MC was partially supported by CMUP, member of LASI, which is financed by national funds through FCT -- Funda\c c\~ao para a Ci\^encia e a Tecnologia, I.P., under the projects with references UIDB/00144/2020 and UIDP/00144/2020. MC also acknowledges financial support from the project PTDC/MAT-PUR/4048/2021.

\def\cprime{$'$}

\end{document}